\documentclass[11pt,twoside]{article}
\usepackage{fullpage}

\usepackage{fullpage, amsthm, amsmath, amssymb, hyperref, dsfont, mathtools,bm}
\usepackage{framed}
\usepackage{caption,subcaption,graphicx,rotating,multirow}
\usepackage{url}
\usepackage{color,xcolor}
\usepackage{enumitem}
\usepackage{epstopdf}

\theoremstyle{plain}
\newtheorem{theorem}{Theorem}

\newtheorem{lemma}[theorem]{Lemma}

\newtheorem{corollary}[theorem]{Corollary}

\newtheorem{proposition}[theorem]{Proposition}
 \newenvironment{proofof}[1]{{\bf {\em Proof of #1.}}}{\hfill \rule{2mm}{2mm} %\qed 
 }
  
\newlength{\widebarargwidth}
\newlength{\widebarargheight}
\newlength{\widebarargdepth}
\DeclareRobustCommand{\widebar}[1]{%
  \settowidth{\widebarargwidth}{\ensuremath{#1}}%
  \settoheight{\widebarargheight}{\ensuremath{#1}}%
  \settodepth{\widebarargdepth}{\ensuremath{#1}}%
  \addtolength{\widebarargwidth}{-0.3\widebarargheight}%
  \addtolength{\widebarargwidth}{-0.3\widebarargdepth}%
  \makebox[0pt][l]{\hspace{0.3\widebarargheight}%
    \hspace{0.3\widebarargdepth}%
    \addtolength{\widebarargheight}{0.3ex}%
    \rule[\widebarargheight]{0.95\widebarargwidth}{0.1ex}}%
  {#1}}
  
\makeatletter
\long\def\@makecaption#1#2{
        \vskip 0.8ex
        \setbox\@tempboxa\hbox{\small {\bf #1:} #2}
        \parindent 1.5em  %% How can we use the global value of this???
        \dimen0=\hsize
        \advance\dimen0 by -3em
        \ifdim \wd\@tempboxa >\dimen0
                \hbox to \hsize{
                        \parindent 0em
                        \hfil 
                        \parbox{\dimen0}{\def\baselinestretch{0.96}\small
                                {\bf #1.} #2
                                } 
                        \hfil}
        \else \hbox to \hsize{\hfil \box\@tempboxa \hfil}
        \fi
        }
\makeatother

\newcommand{\half}{\frac{1}{2}}

\newcommand{\normal}{\ensuremath{\mathcal{N}}}

\newcommand{\bigo}{\ensuremath{\mathcal{O}}}

\renewcommand{\P}{\operatorname{\mathbb{P}}}
\newcommand{\E}{\operatorname{\mathbb{E}}}

\newcommand{\R}{\mathbb{R}}

\newcommand{\Rp}{\mathbb{R}_{+}}

\newcommand{\indi}{\mathds{1}}

\newcommand{\imnb}{\mathbf{i}}

\newcommand{\rmd}{\mathrm{d}}

\newcommand{\twonorm}[1]{\left\|#1\right\|_{\ell_2}}

\newcommand{\abs}[1]{\left|#1\right|}
\newcommand{\vct}[1]{\bm{#1}}
\newcommand{\mtx}[1]{\bm{#1}}

\def\BC{\begin{center}}
\def\EC{\end{center}}
\def\BIT{\begin{itemize}}
\def\EIT{\end{itemize}}
\def\BET{\begin{enumerate}}
\def\EET{\end{enumerate}}
\def\BEQ{\begin{equation}}
\def\EEQ{\end{equation}}

\long\def\comment#1{}

\usepackage{natbib}
\newcommand{\erfc}{\mathsf{erfc}}
\newcommand{\err}{\mathsf{err}}
\newcommand{\A}{\mathcal{A}}
\newcommand{\D}{\mathcal{D}}

\newcommand{\KL}{\mathsf{KL}} 
\newcommand{\TV}{\mathsf{TV}} 
\newcommand{\RS}{\mathsf{RS}} 
\newcommand{\oneRSB}{\mathsf{1RSB}} 

\newcommand{\x}{\textup{\texttt{x}}}

\begin{document}
    
\title{\bf{\LARGE{Finite Size Corrections and Likelihood Ratio Fluctuations in the Spiked Wigner Model}}}
\author{Ahmed El Alaoui\thanks{Department of EECS, UC Berkeley, CA. Email: elalaoui@berkeley.edu}
\and
Florent Krzakala\thanks{Laboratoire de Physique Statistique, CNRS, PSL Universit\'es \& Ecole Normale Sup\'erieure, Sorbonne Universit\'es et Universit\'e Pierre \& Marie Curie, Paris, France.} 
\and
Michael I. Jordan\thanks{Departments of EECS and Statistics, UC Berkeley, CA.}
}

\date{}
\maketitle

\vspace*{-.3in} 
\begin{abstract}
In this paper we study principal components analysis in the regime of high dimensionality and high noise.
Our model of the problem is a rank-one deformation of a Wigner matrix where the signal-to-noise ratio (SNR)
is of constant order, and we are interested in the fundamental limits of detection of the spike.
Our main goal is to gain a fine understanding of the asymptotics for the log-likelihood ratio process,
also known as the free energy, as a function of the SNR. Our main results are twofold. We first prove
that the free energy has a \emph{finite-size correction} to its limit---the \emph{replica-symmetric formula}---which
we explicitly compute. This provides a formula for the Kullback-Leibler divergence between the planted and
null models. Second, we prove that below \emph{the reconstruction threshold}, where it becomes impossible
to reconstruct the spike, the log-likelihood ratio has fluctuations of constant order and converges in
distribution to a Gaussian under both the planted and (under restrictions) the null model. As a consequence,
we provide a general proof of contiguity between these two distributions that holds up to the reconstruction
threshold, and is valid for an arbitrary separable prior on the spike. Formulae for the total variation distance,
and the Type-I and Type-II errors of the optimal test are also given. Our proofs are based on Gaussian
interpolation methods and a rigorous incarnation of the cavity method, as devised by Guerra and Talagrand
in their study of the Sherrington-Kirkpatrick spin-glass model.
     
\end{abstract}

\section{Introduction}
\label{sxn:intro}
Spiked models, which are distributions over matrices of the form ``signal + noise," have been a
mainstay in the statistical literature since their introduction by~\cite{johnstone2001distribution}
as models for the study of high-dimensional principal component analysis that are tractable yet realistic.
Spectral properties of these models have been extensively studied, in particular in random matrix theory,
where they are known as \emph{deformed ensembles}~\citep{peche2014ICM}. Landmark investigations in this
area~\citep{baik2005phase,baik2006eigenvalues,peche2006largest,feral2007largest,capitaine2009largest}
have established the existence of a \emph{spectral threshold} above which the top eigenvalue detaches
from the bulk of eigenvalues and becomes informative about the spike, and below which the top eigenvalue
bears no information. Estimation using the top eigenvector undergoes the same transition, where it is known 
to ``lose track" of the spike below the spectral threshold~\citep{paul2007asymptotics, nadler2008finite,
johnstone2009consistency, benaych2011eigenvalues}. Although these spectral analyses have provided many
insights, as have analyses based on more thoroughgoing usage of spectral data and/or more advanced
optimization-based procedures~\citep[see][and references therein]{amini2008high,berthet2013optimal,
onatski2013asymptotic,onatski2014signal,dobriban2017sharp}, they stop short of characterizing the
fundamental limits of estimating the spike, or detecting its presence from the observation of a
sample matrix. These questions, information-theoretic and statistical in nature, are more naturally
approached by looking at objects such as the posterior law of spike and the associated likelihood
ratio process.    

The main approach to date to the challenging problem of controlling the likelihood ratio
is via the \emph{second moment method}. Controlling the second moment enables one to
show contiguity, in the sense of~\cite{lucien1960locally}, between the planted and null models and
thus declare impossibility of \emph{strong} detection---i.e., the impossibility of vanishing Type-I and
Type-II errors of any given test---in the region where this second moment is
bounded~\citep{banks2017information,perry2016optimality}. This method is known, however, to require careful
conditioning and truncation due to the existence of rare but catastrophic events under which the likelihood
ratio becomes exponentially large. These events thus dominate the second moment, although they are virtually
irrelevant to the detection task. Moreover, even after conditioning the method may fail in identifying
the detection thresholds, depending on the structure of the spike. Furthermore, contiguity has little or no
bearing on the problem of \emph{weak} detection: When errors are inevitable, what is the smallest
error achievable by any test?

Motivated by a desire to overcome these limitations, we consider a particularly simple spiked
model---the \emph{rank-one spiked Wigner model}---and provide an alternative approach to the
detection problem that obviates the use of the second moment method altogether. This is achieved
by obtaining asymptotic distributional results for the log-likelihood ratio process, then appealing
to standard results from the theory of statistical experiments. We are thereby able to provide
solutions to both the strong and weak variants of the detection problem.
To study the likelihood ratio in this setting we build on the technology developed by Aizenman,
Guerra, Panchenko, Talagrand, and many others, in their study of the Sherrington-Kirkpatrick
(SK) spin glass-model.  Specifically, we make use of Gaussian interpolation methods and Talagrand's
cavity method.  

\subsection{Setup and summary of the results}%
In the spiked Wigner model, one observes a rank-one deformation of a Wigner matrix ${\bm W}$:  
\begin{equation}\label{spiked_wigner_model}
\mtx{Y} = \sqrt{\frac{\lambda}{N}} \vct{x}^* \vct{x}^{*\top} + \mtx{W},
\end{equation}
where $W_{ij} = W_{ji} \sim \normal(0,1)$ and $W_{ii} \sim \normal(0,2)$ are independent for all $1 \le i \le j \le N$. The \emph{spike} vector $\vct{x}^* \in \R^N$ represents the signal to be recovered, or its presence detected.  
We assume that the entries $x^*_i$ of the spike are drawn i.i.d.\ from a prior distribution $P_{\x}$ on $\R$ having \emph{bounded} support. The parameter $\lambda \ge 0$ plays the role of the signal-to-noise ratio, and the scaling by $\sqrt{N}$ is such that the signal and noise components of the observed data
are of comparable magnitudes. This places the problem in a high-noise regime where consistency is not possible but partial recovery still is. 
As a matter of convenience, we discard the diagonal terms $Y_{ii}$ from the observations.  (Adding the diagonal back does not pose any additional technical difficulties, and our results can be straightforwardly extended to this case.)          
We endow the real line with the Borel $\sigma$-algebra and define $P_{\x}$ on it. 
We denote by $\P_{\lambda}$ the joint probability law of the observations $\mtx{Y}=\{Y_{ij} : 1\le i < j \le N\}$ as per~\eqref{spiked_wigner_model} and define the likelihood ratio
\begin{equation}\label{likelihood_ratio}
L(\mtx{Y};\lambda) :=  \frac{\mathrm{d}\mathbb{P}_{\lambda}}{\mathrm{d}\mathbb{P}_{0}}(\mtx{Y}).
\end{equation}
A simple computation based on conditioning on $\vct{x}^*$ reveals that
\begin{equation}\label{likelihood_ratio_explicit}
L(\mtx{Y};\lambda) = \int \exp\Big(\sqrt{\frac{\lambda}{N}}\sum_{i < j} Y_{ij}x_ix_j -\frac{\lambda}{2N}\sum_{i < j} x_i^2x_j^2\Big) ~\rmd P_{\x}^{\otimes N}(\vct{x}).
\end{equation}
We define the free energy (density) associated with the model $\P_{\lambda}$ to be
\begin{equation}\label{free_energy}
F_N := \frac{1}{N} \E_{\P_{\lambda}}\log L(\mtx{Y};\lambda). 
\end{equation}
We see that $F_N = \frac{1}{N} D_{\KL}(\P_\lambda, \P_0)$, where $D_{\KL}$ is the Kullback-Leibler divergence
between probability measures. (The free energy is usually defined differently in the literature as the
log-normalizing constant in the posterior of $\vct{x}^*$ given $\mtx{Y}$. The two definitions are strictly
equivalent.) It was initially argued via heuristic replica and cavity computations~\citep{lesieur2015phase,lesieur2017constrained}
that $F_N$ converges to a limit $\phi_{\RS}(\lambda)$, which is referred to as the
\emph{replica-symmetric formula}. This formula, variational in nature, encodes in principle a full characterization
of the limits of estimating the spike with non-trivial accuracy. Indeed, various formulae for other
information-theoretic quantities can be deduced from it, including the mutual information between
$\vct{x}^*$ and $\mtx{Y}$, the minimal mean squared error of estimating $\vct{x}^*$ based on $\mtx{Y}$,
and the overlap $|\vct{x}^\top\vct{x}^*|/N$ of a draw $\vct{x}$ from the posterior $\P_\lambda(\cdot | \mtx{Y})$
with the spike $\vct{x}^*$. Most of these claims have subsequently been proved rigorously in a series of
papers~\citep{deshpande2014information,deshpande2016asymptotic,barbier2016mutual,krzakala2016mutual,lelarge2016fundamental} under various assumptions on the prior.  However, these results stop short
of providing explicit characterizations of thresholds for the detection problem. 

The main goal of this paper is to gain a more refined understanding of the asymptotic behavior of the log-likelihood ratio $\log L(\mtx{Y};\lambda)$, and its mean $NF_N$, under $\P_{\lambda}$ as $N$ becomes large. We first determine \emph{the finite-size correction} of $F_N$ to its limit $\phi_{\RS}(\lambda)$: we prove (under conditions on $P_{\x}$) that $N(F_N - \phi_{\RS}(\lambda))$ converges to a limit $\psi_{\RS}(\lambda)$ with rate $\bigo(1/\sqrt{N})$. Besides providing an explicit rate of convergence of $F_N$ to its limit, this result translates into a formula for the Kullback-Leibler divergence $D_{\KL}$, which is particularly interesting below the reconstruction threshold: we will see that in this regime $\phi_{\RS}(\lambda) = 0$, so $D_{\KL}$ ceases to be extensive in the size of the system and converges to a finite value $\psi_{\RS}(\lambda)$. 

%\vspace{.1cm}
Second, we prove that in this same regime, the log-likelihood ratio $\log L(\mtx{Y};\lambda)$ has fluctuations of constant order under $\P_{\lambda}$, and converges asymptotically to a Gaussian with a mean equal to half the variance.
This allows us to provide an alternative proof of contiguity between $\P_{\lambda}$ and $\P_{0}$, valid in the
entire regime where contiguity can possibly hold, as well as a formula for the Type-II error for testing
between these two distributions.   
%\vspace{.1cm}

Under the null distribution $\P_0$ on the other hand, the model is equivalent to the widely studied
Sherrington-Kirkpatrick model (provided that $P_{\x} = \half \delta_{-1} + \half \delta_{+1}$).
In one of the first rigorous results on this model, \cite{aizenman1987some}
proved that in the high-temperature regime and in the absence of an external field, the fluctuations of
the log-partition function of the model about its mean, which is given by the ``annealed" computation,
are asymptotically Gaussian with explicit mean and variance. By mapping their result into our setting,
we obtain the fluctuations of $\log L(\mtx{Y};\lambda)$ under $\P_{0}$ in the non-reconstruction phase,
as well as a formula for the Type-I error. Although we only obtain this last formula for the Rademacher
prior, we conjecture its validity for arbitrary priors. An interesting symmetry emerges from these
results: the limiting Gaussians under $\P_{\lambda}$ and $\P_{0}$ have means of equal magnitude and
opposite signs, and equal variances. This symmetry causes the Type-I and Type-II errors to be equal.  
Adding up the two latter quantities, we obtain a formula for the total variation distance
$D_{\TV}(\P_{\lambda},\P_{0})$. 

Our results are in the spirit of those of~\cite{onatski2013asymptotic,onatski2014signal}, who studied 
the likelihood ratio of the joint eigenvalue densities under the spiked covariance model with a sphericity prior, 
and showed its asymptotic normality below the spectral threshold. 
Their results thus pertain to eigenvalue-based tests while ours are for arbitrary tests, albeit for a simpler model.   
Our results show in particular that it is still possible to distinguish the planted model
$\P_{\lambda}$ from the null model $\P_{0}$ with non-vanishing probability below the reconstruction
threshold; i.e., even when estimation of the spike $\vct{x}^*$ becomes impossible. Performing such a test
in practice of course hinges on the computational problem of efficiently computing the likelihood
ratio; we leave open this question of constructing computationally efficient tests in the non-reconstruction
phase.\footnote{If one observes the diagonal, then one can test using the trace of $\mtx{Y}$.
This test would still however be suboptimal.}

%%%%%%%%%%%%%%%%%%%%%%%%%%%%%%%%%%%
\subsection{Background}
\paragraph{The $\RS$ formula.} For $r \ge 0$, consider the function
\begin{equation}\label{log_partition_scalar_gaussian_channel}
\psi(r) := \E_{x^*,z} \log \int \exp\left(\sqrt{r}zx + r xx^* - \frac{r}{2} x^2\right) \rmd P_{\x}(x),
\end{equation}
where $z \sim\normal(0,1)$, and $x^* \sim P_{\x}$. 
This is the $\KL$ divergence between the distributions of the random variables $y = \sqrt{r}x^*+z$ and $z$. We define the Replica-Symmetric ($\RS$) potential 
\begin{equation}\label{RS_potential}
F(\lambda, q) := \psi(\lambda q) -\frac{\lambda q^2}{4},
\end{equation}
and finally define the $\RS$ formula
\begin{equation}\label{RS_formula}
\phi_{\RS}(\lambda) := \sup_{q \ge 0}~ F(\lambda, q).
\end{equation}
A central result in this context is that free energy $F_N$ converges to the $\RS$ formula for all $\lambda \ge 0$ ~\citep{lesieur2015phase,lesieur2017constrained,deshpande2016asymptotic,barbier2016mutual,krzakala2016mutual,lelarge2016fundamental}:    
\[F_{N} ~~ \longrightarrow ~~\phi_{\RS}(\lambda).\] 
The values of $q$ that maximize the $\RS$ potential and their properties play an important role in the theory. \cite{lelarge2016fundamental}    
proved that the map $q \mapsto F(\lambda,q)$ has a unique maximizer $q^* = q^*(\lambda)$ for all $\lambda \in \D$ where $\D$ is the set of points where the function $\lambda \mapsto \phi_{\RS}(\lambda)$ is differentiable. 
By convexity of $\phi_{\RS}$ (see next section), $\D = \R_+ \setminus \mbox{countable set}$. Moreover, they showed that the map $ \lambda \in \D \mapsto q^*(\lambda)$ is non-decreasing, and 
\begin{equation}\label{limits_q_star}
\lim_{\underset{\lambda \in \D}{\lambda \to 0}} q^*(\lambda) = \E_{P_{\x}}[X]^2, \qquad \text{and} \qquad \lim_{\underset{\lambda \in \D}{\lambda \to \infty}} q^*(\lambda) = \E_{P_{\x}}[X^2].
\end{equation}
One should interpret the value $q^*(\lambda)$ as the best overlap an estimator $\widehat{\theta}(\mtx{Y})$ based on observing $\mtx{Y}$ can have with the spike $\vct{x}^*$. Indeed, the overlap $\big|\vct{x}^\top\vct{x}^*\big|/N$ between the spike $\vct{x}^*$ and a random draw $\vct{x}$ from the posterior $\P_\lambda(\cdot | \mtx{Y})$ should concentrate in the large $N$ limit about $q^*(\lambda)$ (hence the name ``replica-symmetry"). A matrix variant of this result (where one estimates $\vct{x}^*\vct{x}^{*\top}$) was proved in~\citep{lelarge2016fundamental}. In Section~\ref{sxn:convergence_overlaps}, we prove strong (vector) versions of this result where under mild assumptions, optimal rates of convergence are given.

\paragraph{The reconstruction threshold.} The first limit in~\eqref{limits_q_star} shows that when the prior $P_{\x}$ is not centered, it is always possible to have a non-trivial overlap with $\vct{x}^*$ for any $\lambda>0$. On the other hand, when the prior has zero mean, and since $q^*$ is a non-decreasing function of $\lambda$, it is useful to define the critical value of $\lambda$ below which estimating $\vct{x}^*$ becomes impossible:
\begin{equation}\label{lambda_critical}
\lambda_c := \sup \big\{ \lambda >0 ~:~ q^*(\lambda) =0\big\}.
\end{equation}
We refer to $\lambda_c$ as the \emph{critical} or \emph{reconstruction} threshold. The next lemma establishes a natural bound on $\lambda_c$.
\begin{lemma} \label{spectral_bound}
We have
\begin{equation}\label{eq:spectral_bound}
\lambda_c \cdot \left(\E_{P_{\x}}[X^2]\right)^2 \le 1.
\end{equation}
\end{lemma}
\begin{proof}
Indeed, assume that $P_{\x}$ is centered, and let $\lambda > (\E[X^2])^{-2}$. Since $\psi'(0) = \half\E_{P_{\x}}[X]^2 = 0$ and $\psi''(0) = \half(\E_{P_{\x}}[X^2])^2$, we see that $\partial_q F(\lambda, 0) = 0$ and $\partial_q^2 F(\lambda,0) = \frac{\lambda}{2}(\lambda \E_{P_{\x}}[X^2]^2 - 1)>0$. So $q=0$ cannot be a maximizer of $F(\lambda,\cdot)$. Therefore $q^*(\lambda)>0$ and $\lambda \ge \lambda_c$.
\end{proof}
The importance of Lemma~\ref{spectral_bound} stems from the fact that the value $\left(\E_{P_{\x}}[X^2]\right)^{-2}$ is the spectral threshold previously discussed. Above this value, the first eigenvalue of the matrix $\mtx{Y}$ 
leaves the bulk, and is at the edge of the bulk below it~\citep{peche2006largest,capitaine2009largest,feral2007largest}. This value also marks the limit below which the first eigenvector of $\mtx{Y}$ captures no information about the spike $\vct{x}^*$~\citep{benaych2011eigenvalues}. Inequality~\eqref{eq:spectral_bound} can be strict or turn into equality depending on the prior $P_{\x}$. For instance, there is equality if the prior is Gaussian or Rademacher---so that the first eigenvector overlaps with the spike as soon as estimation becomes possible at all---and strict inequality in the case of the (sufficiently) sparse Rademacher prior $P_{\x} = \frac{\rho}{2}\delta_{-1/\sqrt{\rho}} + (1-\rho)\delta_{0} + \frac{\rho}{2}\delta_{+1/\sqrt{\rho}}$. More precisely, there exists a value 
\[\rho^* = \inf \big\{\rho \in (0,1) ~:~ \psi'''(0)<0 \big\} \approx 0.092,\]
such that $\lambda_c = 1$ for $\rho \ge \rho^*$, and $\lambda_c < 1$ for $\rho < \rho^*$. 
In the latter case, the spectral approach to estimating $\vct{x}^*$ fails for $\lambda \in (\lambda_c, 1)$, and it is believed that no polynomial time algorithm succeeds in this region~\citep{lesieur2015phase,krzakala2016mutual,banks2017information}.

%%%%%%%%%%%%%%%%%%%%%%%%%%%%%%%%%%%%%%%
\section{Main results}
\label{sxn:main_results}

\subsection{Finite size corrections to the $\RS$ formula}

The results we are about to present hold in a possibly slightly smaller set than $\D$. While uniqueness of $q^*$ only needs first differentiability of the $\RS$ formula, our results need a second derivative to exist. In physics parlance, our results do not hold at values of $\lambda$ at which a particular kind of \emph{first-order phase transition} occurs, namely, one in which the order parameter $q^*$ is not differentiable. The presence of these transitions depends again on the prior $P_{\x}$. For the Gaussian and Rademacher prior, there are no such transitions, while for the sparse Rademacher prior discussed above, there is \emph{one} first-order transition where $q^{*'}$ is not defined for every $\rho < \rho^*$.   
Thus we define the set
\[\A = \big\{\lambda > 0 ~:~ \phi_{\RS} \mbox{ is twice differentiable at } \lambda.\big\}.\]
Since $\phi_{\RS}$ is the point-wise limit of a sequence $(F_N)$ of convex functions, it is also convex. Then by Alexandrov's theorem~\citep{aleksandorov1939almost}, the set $\A$ is of full Lebesgue measure in $\R_+$ (cf.\ $\D = \R_+ \setminus \mbox{countable set}$.) Moreover, we can see that  $(0,\lambda_c) \subset \A$, since if $\lambda \in \A \cap (0,\lambda_c)$, we have $q^*(\lambda) = 0$, therefore $\phi_{\RS}(\lambda) = 0$. By continuity, $\phi_{\RS}$ vanishes on the entire interval $(0,\lambda_c)$.
Our first main result is to establish the existence of a function $\lambda \mapsto \psi_{\RS}(\lambda)$ defined on $\A$ such that either below $\lambda_c$ or above it when the prior $P_{\x}$ is not symmetric about the origin, we have
\[N(F_N - \phi_{\RS}(\lambda)) \longrightarrow \psi_{\RS}(\lambda).\]
An explicit formula for $\psi_{\RS}$ will be given. But first we need to introduce some notation. Let $\lambda \in \A$ and consider the quantities 
\begin{align}\label{4_replica_overlaps}
a(0) = \E \left[\langle x^2\rangle_r^2\right] - q^{*2}(\lambda), \quad
a(1) = \E \left[\langle x^2 \rangle_r \langle x \rangle_r^2\right]  - q^{*2}(\lambda), \quad
a(2) = \E \left[\langle x\rangle_r^4\right]  - q^{*2}(\lambda),
\end{align} 
where 
\[\langle \cdot \rangle_r = \frac{\int \cdot \exp\left(\sqrt{r}zx + r xx^* - \frac{r}{2} x^2\right) \rmd P_{\x}(x)}{\int \exp\left(\sqrt{r}zx + r xx^* - \frac{r}{2} x^2\right) \rmd P_{\x}(x)},\]
with $r = \lambda q^*(\lambda)$ and the expectation operator $\E$ is w.r.t.\ $x^* \sim P_{\x}$ and $z \sim \normal(0,1)$. The Gibbs measure $\langle \cdot \rangle_r$ can be interpreted as the posterior distribution of $x^*$ given the observation $y = \sqrt{r}x^*+z$. (More on this point of view in Section~\ref{sxn:convergence_overlaps}.)
Now let
%\begin{align}\label{eigenvalues}
%\mu_1(\lambda) = \lambda (a(0) -2a(1)+a(2)), \quad
%\mu_2(\lambda) = \lambda (a(0) -3a(1)+2a(2)),
%\end{align}
\begin{gather}\label{eigenvalues}
\begin{aligned}
\mu_1(\lambda) &= \lambda (a(0) -2a(1)+a(2)),\\
\mu_2(\lambda) &= \lambda (a(0) -3a(1)+2a(2)),
\end{aligned}
\end{gather}
and finally define
\begin{align}\label{sigma_RS}
\psi_{\RS}(\lambda) := \frac{1}{4} \left(\log(1-\mu_1) - 2\log(1-\mu_2) + \lambda\frac{4a(1)-3a(2)}{1-\mu_1} - \lambda a(0)\right).
\end{align}
We will prove (Lemma~\ref{latala_guerra}) that $\mu_2 \le \mu_1 <1$ for all $\lambda \in \A$ so that this function is well defined on $\A$. 
\begin{theorem}\label{finite_size_correction}
For $\lambda \in \A$, if either $\lambda<\lambda_c$, or $\lambda > \lambda_c$ and the prior $P_{\x}$ is not symmetric about the origin, then
\[ N\big(F_N - \phi_{\mathsf{RS}}(\lambda)\big) = \psi_{\RS}(\lambda) + \bigo\Big(\frac{1}{\sqrt{N}}\Big),\]
or equivalently, $D_{\KL}(\P_\lambda,\P_0)  = N\phi_{\RS}(\lambda) + \psi_{\RS}(\lambda) + \bigo(1/\sqrt{N})$.
\end{theorem}
The theorem asserts that either below the reconstruction threshold, or above it when the prior $P_{\x}$ is not symmetric, the free energy $F_N$ has a finite-size correction of order $1/N$ to its limit $\phi_{\RS}$ and a subsequent term of order $N^{-3/2}$ in the expansion. In the case $\lambda > \lambda_c$ with symmetric prior, the problem is invariant under a sign flip of the spike, so the overlap $\vct{x}^\top \vct{x}^*/N$ has a symmetric distribution, and hence concentrates equiprobably about \emph{two distinct} values $\pm q^*(\lambda)$. Our techniques do not survive this symmetry, and resolving this case seems to require a new approach.
  
We see that $D_{\KL}(\P_\lambda,\P_0)$ is an extensive quantity in $N$ whenever $\phi_{\RS}(\lambda) >0$, or equivalently, $\lambda>\lambda_c$. On the other hand, this $\KL$ is of constant order below $\lambda_c$: 

\vspace{-.3cm}
\paragraph{Centered prior.} Let us consider the case where the prior $P_{\x}$ has zero mean, and unit variance (the latter can be assumed without loss of generality by rescaling $\lambda$), so that Lemma~\ref{spectral_bound} reads $\lambda_c \le 1$. If $\lambda < \lambda_c$, we have $q^*(\lambda) = 0$, $\phi_{\RS}(\lambda) = 0$, and one can check that in this case
\[a(0) = (\E_{P_{\x}}[X^2])^2 = 1, \quad a(1) = \E_{P_{\x}}[X^2]\E_{P_{\x}}[X]^2  = 0, \quad a(2) = \E_{P_{\x}}[X]^4 = 0.\]
Therefore, expression~\eqref{sigma_RS} simplifies to 
\[\psi_{\RS}(\lambda) = \frac{1}{4}\left(-\log\left(1-\lambda \right) - \lambda \right).\] 
By the above calculation, we have a formula for the $\KL$ divergence between $\P_{\lambda}$ and $\P_{0}$ below the reconstruction threshold $\lambda_c$ (see plot in Figure~\ref{minimal_error_KL_TV}):
\begin{corollary}\label{cor_kul_divergence}
Assume the prior $P_{\x}$ is centered and of unit variance. Then for all $\lambda <\lambda_c$,
\begin{equation}\label{kul_divergence}
D_{\KL}(\P_\lambda,\P_0) = \frac{1}{4}\left(-\log\left(1-\lambda\right) - \lambda \right) + \bigo\Big(\frac{1}{\sqrt{N}}\Big).
\end{equation}
\end{corollary}

\paragraph{More information on $\psi_{\RS}$.} Expression~\eqref{sigma_RS} looks mysterious at first sight. Let us briefly explain its origin. A slightly less processed expression for $\psi_{\RS}$ is the following  
\begin{align*}
\psi_{\RS}(\lambda) = \frac{1}{4}\int_0^1 \left(-\frac{\mu_1}{1-t\mu_1} + \frac{2\mu_2}{1-t\mu_2} + \lambda\frac{4a(1) - 3a(2)}{(1-t\mu_1)^2}\right)\rmd t - \frac{\lambda}{4} a(0),
\end{align*}
after which~\eqref{sigma_RS} follows by simple integration. The integrand in the above expression is obtained, as we will show, as the first entry $z(0)$ of the solution $\vct{z}= [z(0),z(1),z(2)]^{\top}$ of the $3 \times 3$ linear system 
\[(\mtx{I}-t\mtx{A})\vct{z} = \vct{a},\]
where $\vct{a} = [a(0), a(1), a(2)]^\top$ and $\mtx{A}$ is the ``cavity" matrix 
\begin{align*}
\mtx{A} := \lambda \cdot
\begin{bmatrix}
a(0) & -2a(1) & a(2)\\
a(1) & a(0) -a(1) -2a(2) & -2a(1) + 3a(2)\\
a(2) & 4a(1) - 6a(2) & a(0)-6a(1)+6a(2)
\end{bmatrix}.
\end{align*}
The above matrix happens to have two eigenvalues which are exactly $\mu_1$ and $\mu_2$. The matrix $\mtx{A}$ and the above linear system will emerge naturally as a result of the cavity method. On the other hand, the integral over the time parameter $t$ is along an interpolation path invented by~\cite{guerra2001sum}, \citep[see also][]{guerra2002thermodynamic} in the context of the Sherrington--Kirkpatrick model, and the integrand can be interpreted as the asymptotic variance in a central limit theorem satisfied by the overlap between two ``replicas" under the law induced by a certain interpolating Gibbs measure. A definition of these notions with the corresponding results can be found in Sections~\ref{sxn:convergence_overlaps} and~\ref{sxn:interpolation_method}. The full execution of the cavity method is relegated to Section~\ref{sxn:cavity_method}.

%%%%%%%%%%%%%%%%%%%%%%%%%%%%%%%%%%%%%
\subsection{Fluctuations below the reconstruction threshold}
Corollary~\ref{kul_divergence} asserts that below the reconstruction threshold, the expectation of the log-likelihood ratio $\log L(\mtx{Y};\lambda)$ under $\P_{\lambda}$ is of constant order (in $N$) and is asymptotically equal to $(-\log\left(1-\lambda\right) - \lambda)/4$. In this section we are interested in the fluctuations of this quantity about its expectation. It can be seen by a standard concentration-of-measure argument that for all $\lambda > 0$, $\log L(\mtx{Y};\lambda)$ concentrates about its expectation with fluctuations bounded by $\bigo(\sqrt{N})$. While this bound is likely to be of the right order above $\lambda_c$~\citep[this is true for the SK model, see][]{guerra2002central}, it is very pessimistic below $\lambda_c$. Indeed, we will show that the fluctuations are of constant order with a Gaussian limiting law in this regime. 
This phenomenon was noticed early on in the case of the SK model: \cite{aizenman1987some} showed that in the absence of an external field, the log-partition function of this model has (shifted) Gaussian fluctuations about its easily computed ``annealed average" in high temperature. We will directly deduce from their result  a central limit theorem for $\log L(\mtx{Y};\lambda)$ under $\P_{0}$ in the case where the prior $P_{\x}$ is Rademacher. Furthermore, a proof by \cite{talagrand2011mean2} of their result provided us with a road map for proving a similar result under $\P_{\lambda}$. We now present our second main result along with consequences for hypothesis testing. % 
For $\lambda <1$, let
\[\mu(\lambda) = \frac{1}{4} \left(-\log(1-\lambda)-\lambda\right), \qquad \mbox{and} \qquad \sigma^2(\lambda) = {2 \mu(\lambda)}.\]
\begin{theorem} \label{central_limit_theorem}
Assume the prior $P_{\x}$ is centered and of unit variance.
\begin{itemize}
\item[(i)] For all $\lambda <\lambda_c$, if $\mtx{Y} \sim \P_{\lambda}$ then
\[\log L(\mtx{Y};\lambda) \rightsquigarrow \normal(\mu,\sigma^2).\]
\item[(ii)] Under the additional condition $P_{\x} = \half \delta_{-1} + \half \delta_{+1}$, if $\mtx{Y} \sim \P_{0}$, then for all $\lambda <1$,
\[\log L(\mtx{Y};\lambda) \rightsquigarrow \normal(-\mu,\sigma^2).\]
\end{itemize}
\end{theorem}
The symbol $``\rightsquigarrow"$ denotes convergence in distribution as $N \to \infty$. The formal connection to the SK model and the proof of the above theorem are presented in Section~\ref{sxn:proof_fluctuations}. 

A remarkable feature is the symmetry between the above two statements. Roughly speaking, this symmetry takes its roots in the fact that the model under the alternative distribution $\P_{\lambda}$ ``lives on the Nishimori line": under $\P_{\lambda}$, which is the spiked model, the interaction of the spike $\vct{x}^*$ with a replica $\vct{x}^{(1)}$ creates terms that account for twice the contribution of the interaction between two independent replicas $\vct{x}^{(1)}$ and $\vct{x}^{(2)}$, and thus flips the sign of the mean from $-\mu$ to $\mu$. This mechanism will become apparent from the proof. Moreover, the fact that the mean is half the variance in these limiting Gaussians has interesting consequences for hypothesis testing. The next subsection is devoted to this problem.
 
We believe that statement \emph{(ii)} is still valid up to $\lambda_c$ for a general prior (of zero mean and unit variance). It is possible to prove the convergence up to \emph{some} value $\lambda_0 \le \lambda_c$ with essentially the same approach as ours, but reaching the optimal threshold seems to require more technical work. In particular, our interpolation bound (see our ``main estimate", section~\ref{sxn:main_estimate}) has to be significantly improved to deal with this case. Progress will be reported in a future work.  

We also point out that similar fluctuation results were recently proved by~\cite{baik2016fluctuations,baik2017fluctuations} for a spherical model where one integrates over the uniform measure on the sphere in the definition of $L(\mtx{Y}; \lambda)$. Their model, due to its integrable nature, is amenable to analysis using tools from random matrix theory. The authors are thus able to also analyze a ``low temperature" regime (absent in our problem) where the fluctuations are no longer Gaussian  but given by the Tracy-Widom distribution. Their techniques seem to be tied to the spherical case however.  

%%%%%%%%%%%%%%%%%%%%%%%%%%%%%%%%%%%%%%%%%
\subsubsection{Strong and weak detection below $\lambda_c$.} 
Consider the problem of deciding whether an array of observations $\mtx{Y}=\{Y_{ij} : 1\le i < j \le N\}$ is likely to have been generated from $\P_{\lambda}$ for a fixed $\lambda >0$ or from $\P_{0}$. Let us denote by $\bm{H}_0 : \mtx{Y} \sim \P_{0}$ the null hypothesis and $\bm{H}_{\lambda}: \mtx{Y} \sim \P_{\lambda}$ the alternative hypothesis. 
Two formulations of this problem exist: one would like to construct a sequence of measurable tests $T: \R^{N(N-1)/2} \mapsto \{0,1\}$ that returns ``0" for ${\bm H}_0$ and ``1" for ${\bm H}_{\lambda}$, for which either
\begin{align}\label{strong_detection}
\lim_{N\to \infty} ~\max ~\Big\{\P_{\lambda}(T(\mtx{Y}) = 0),~ \P_{0}(T(\mtx{Y}) = 1) \Big\} = 0, 
\end{align}
or less stringently, the total mis-classification error, or risk
\begin{equation}\label{misclassif_error}
\err(T) := \P_{\lambda} (T(\mtx{Y}) = 0) + \P_{0} (T(\mtx{Y}) = 1)
\end{equation}
is minimized among all possible tests $T$. The question of existence of a test that answers to the requirement~\eqref{strong_detection} is referred to as the \emph{strong detection} problem, and the question of minimizing the criterion~\eqref{misclassif_error} is referred to as the \emph{weak detection}, or simply \emph{hypothesis testing} problem.  

\paragraph{Strong detection.}
Using a second moment argument based on the computation of a truncated version of $\E L(\mtx{Y};\lambda)^2$, \cite{banks2017information} and \cite{perry2016optimality} showed that $\P_{\lambda}$ and $\P_{0}$ are mutually contiguous when $\lambda < \lambda_0$, where the latter quantity equals $\lambda_c$ for some priors $P_{\x}$ while it is suboptimal for others (e.g., the sparse Rademacher case, see discussion below). 
It is easy to see that contiguity implies impossibility of strong detection since for instance, if $\P_{0}(T(\mtx{Y})=1) \to 0$ then $\P_{\lambda}(T(\mtx{Y})=0) \to 1$.  
Here we show that Theorem~\ref{central_limit_theorem} provides a more powerful approach to contiguity: 
\begin{corollary}  
Assume the prior $P_{\x}$ is centered and of unit variance. Then for all $\lambda < \lambda_c$, $\P_{\lambda}$ and $\P_{0}$ are mutually contiguous. 
\end{corollary}
\begin{proof}
A consequence of statement \emph{(i)} in Theorem~\ref{central_limit_theorem} is that if
\[\frac{\rmd \P_{0}}{ \rmd \P_{\lambda}} ~{\rightsquigarrow}~ U\]
under $\P_{\lambda}$ along some subsequence and for some random variable $U$, then by the continuous mapping theorem we necessarily have 
\[U = \exp \normal(-\mu,\sigma^2).\]
We have $\Pr(U>0) = 1$, and since $\mu = \half \sigma^2$, we have $\E U  = 1$. We now conclude using Le Cam's first lemma in both directions~\cite[Lemma 6.4 or Example 6.5,][]{vandervaart2000asymptotic}. 
\end{proof}
This approach allows one to circumvent second moment computations which are not guaranteed to be tight in general, and necessitate careful and prior-specific conditioning that truncates away undesirable events.  
  
 We note that in the case of the sparse Rademacher prior $P_{\x} = \frac{\rho}{2}\delta_{-1/\sqrt{\rho}} +   (1-\rho)\delta_{0} + \frac{\rho}{2}\delta_{+1/\sqrt{\rho}}$, contiguity holds for all $\lambda < 1$ as soon as $\rho \ge \rho^* \approx 0.092$ by the above corollary, thus closing the gaps in the results of~\cite{banks2017information} and~\cite{perry2016optimality}. Indeed, as argued below Lemma~\ref{spectral_bound}, the reconstruction and spectral thresholds are equal ($\lambda_c=1$) for all $\rho \ge \rho^*$, and differ ($\lambda_c < 1$) below $\rho^*$. This implies that strong detection is impossible for $\lambda <1$ and possible otherwise when $\rho \ge \rho^*$, while it becomes impossible only below $\lambda_c $ but possible otherwise when $\rho < \rho^*$.  

\paragraph{Weak detection.} We have seen that strong detection is possible if and only if $\lambda > \lambda_c$. It is then natural to ask whether weak detection is possible below $\lambda_c$, i.e., is it possible to test with accuracy \emph{better than that of a random guess} below the reconstruction threshold? The answer is \emph{yes}, and this is another consequence of Theorem~\ref{central_limit_theorem}. More precisely, the optimal test minimizing the risk~\eqref{misclassif_error} is the likelihood ratio test which rejects the null hypothesis $\bm{H}_0$ (i.e., returns ``1") if $L(\mtx{Y};\lambda) >1$, and its error is
\begin{equation}\label{optimal_error}
\err^*(\lambda) = \P_{\lambda} (L(\mtx{Y};\lambda) \le 1) + \P_{0} (L(\mtx{Y};\lambda) >1) = 1 - D_{\TV}(\P_\lambda,\P_0).
\end{equation}
One can readily deduce from Theorem~\ref{central_limit_theorem} the Type-I and Type-II errors of the likelihood ratio test:  for all $\lambda < \lambda_c$ the Type-II error is  
\[\P_{\lambda}(\log L(\mtx{Y}; \lambda) \le 0) = \int_{-\infty}^{0} \frac{1}{\sqrt{2\pi}}e^{-(t-\mu)^2/2\sigma^2} \mathrm{d}t + o_N(1) = \half\erfc\left(\frac{\sqrt{\mu}}{2}\right) + o_N(1),\]
and in the case of the Rademacher prior, the Type-I error is
\[\P_{0}(\log L(\mtx{Y}; \lambda) > 0) = \int_{0}^{+\infty}\frac{1}{\sqrt{2\pi}}e^{-(t+\mu)^2/2\sigma^2} \mathrm{d}t + o_N(1) = \half\erfc\left(\frac{\sqrt{\mu}}{2}\right) + o_N(1)\]
for all $\lambda <1$. Here, $\erfc(x) = \frac{2}{\sqrt{\pi}} \int_{x}^\infty e^{-t^2}\rmd t$ is the complementary error function.
These can be combined into a formula for $\err^*(\lambda)$ and the total variation distance between $\P_{\lambda}$ and $\P_{0}$ (see plot in Figure~\ref{minimal_error_KL_TV}): 
\begin{corollary} 
Assume $P_{\x} = \half \delta_{-1} + \half \delta_{+1}$. For all $\lambda <1$, we have
\begin{equation}\label{total_variation}
\lim_{N\to \infty} \err^*(\lambda) = 1-\lim_{N\to \infty} D_{\TV}(\P_{\lambda},\P_{0}) = \erfc\left(\frac{\sqrt{\mu(\lambda)}}{2}\right).
\end{equation}
\end{corollary}
We similarly conjecture that the formula for Type-I error, hence formula~\eqref{total_variation}, should be correct up to $\lambda_c$ for all (bounded) priors with zero mean and unit variance. 
\begin{figure}
\centering
\includegraphics[width=7.5cm]{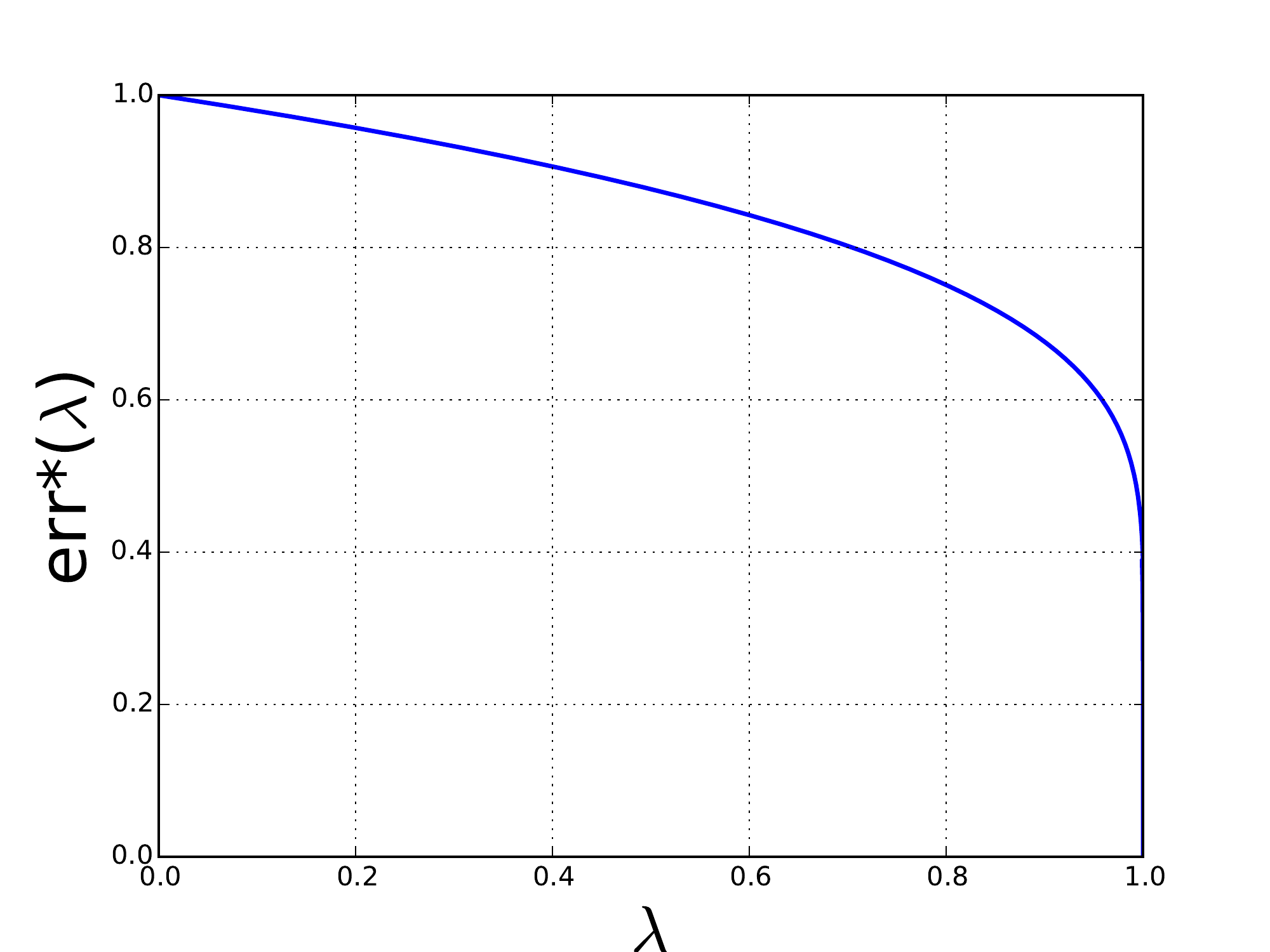}
\includegraphics[width=7.5cm]{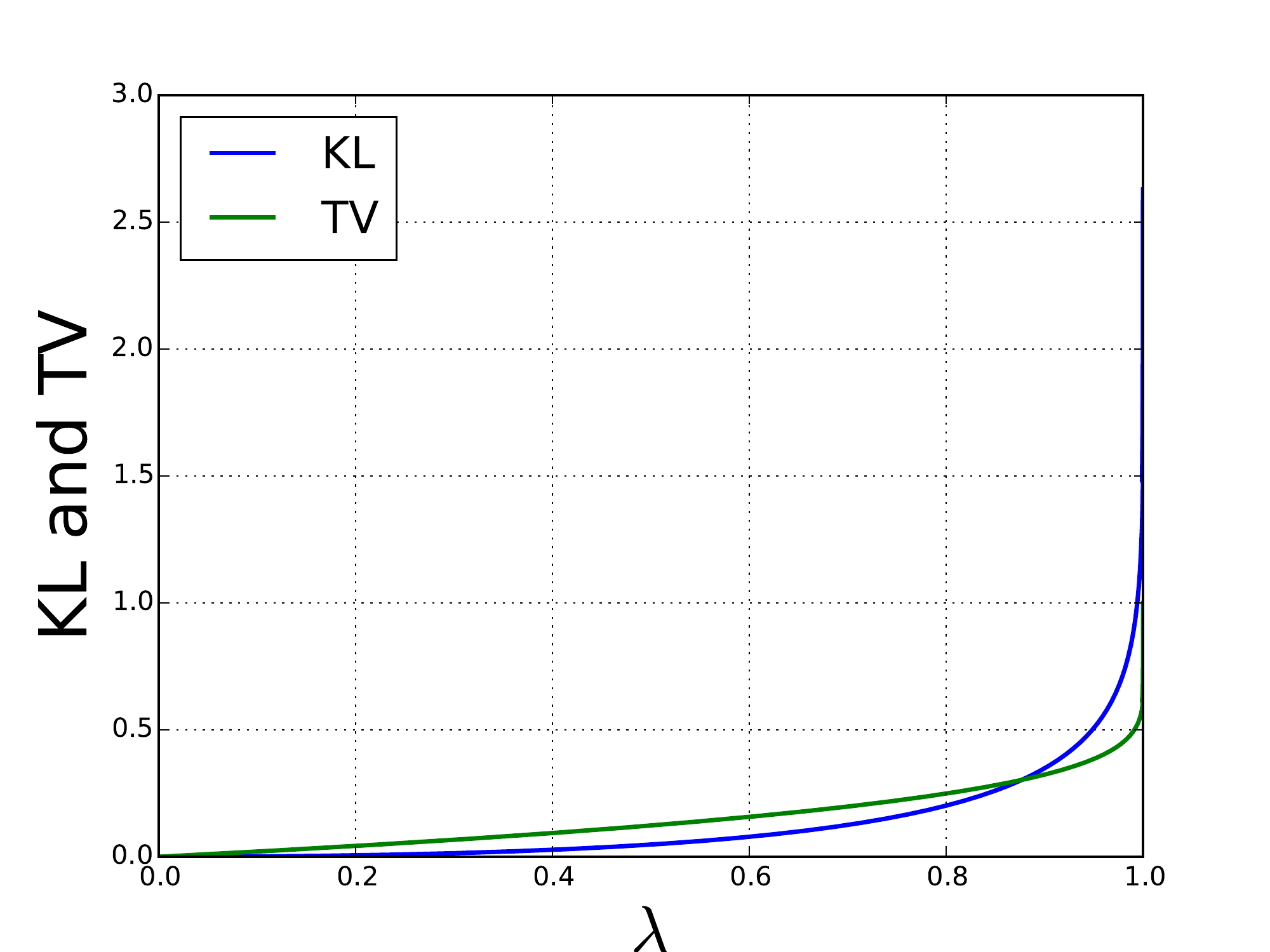}
\caption{Plots of formulas~\eqref{total_variation} and~\eqref{kul_divergence}.}
\protect\label{minimal_error_KL_TV}
\end{figure}

%%%%%%%%%%%%%%%%%%%%%%%%%%%%%%%%%%%%
\section{Overlap convergence: optimal rates}
\label{sxn:convergence_overlaps}
A crucial component of proving our main results is understanding  the rate of convergence of the overlap $\vct{x}^\top\vct{x}^*/N$, where $\vct{x}$ is drawn from $\P_{\lambda}(\cdot|\mtx{Y})$, to its limit $q^*(\lambda)$. By Bayes' rule, we see that
\begin{equation}\label{posterior}
\rmd\P_{\lambda} (\vct{x}| \mtx{Y}) = \frac{ e^{- H(\vct{x})} \rmd P_{\x}^{\otimes N}(\vct{x})}{\int  e^{- H(\vct{x})}  \rmd P_{\x}^{\otimes N}(\vct{x})},
\end{equation} 
where $H$ is the Hamiltonian
\begin{align}\label{hamiltonian}
-H(\vct{x}) &:= -\frac{\lambda}{2N}\sum_{i < j} x_i^2x_j^2 + \sqrt{\frac{\lambda}{N}}\sum_{i < j} Y_{ij}x_ix_j \\
&= -\frac{\lambda}{2N}\sum_{i < j} x_i^2x_j^2 + \sqrt{\frac{\lambda}{N}}\sum_{i < j} W_{ij}x_ix_j + \frac{\lambda}{N}\sum_{i < j}x_ix_i^*x_jx_j^*.\nonumber
\end{align} 
From the formulas~\eqref{likelihood_ratio_explicit} and~\eqref{free_energy}, it is straightforward to see that
\[F_{N} = \frac{1}{N} \E \log~ \int e^{- H(\vct{x})} \rmd P_{\x}^{\otimes N}(\vct{x}),\]
This provides another way of interpreting $F_N$ as the expected log-partition function (or normalizing constant) of the posterior $\P_{\lambda}(\cdot | \mtx{Y})$.  
For an integer $n\ge 1$ and $f : (\R^{N})^{n+1} \mapsto \R$, we define the Gibbs average of $f$ w.r.t.\ $H$ as 
 \begin{equation}\label{gibbs_average_one}
 \left\langle f(\vct{x}^{(1)},\cdots,\vct{x}^{(n)},\vct{x}^*)\right\rangle:= \frac{\int f(\vct{x}^{(1)},\cdots,\vct{x}^{(n)},\vct{x}^*) \prod_{l=1}^n e^{- H(\vct{x}^{(l)})} \rmd P_{\x}^{\otimes N}(\vct{x}^{(l)})}{\int \prod_{l=1}^n e^{- H(\vct{x}^{(l)})}  \rmd P_{\x}^{\otimes N}(\vct{x}^{(l)})}.
 \end{equation}
This is nothing else that the average of $f$ with respect to $\P_{\lambda}(\cdot | \mtx{Y})^{\otimes n}$.
 The variables $\vct{x}^{(l)}, l=1\cdots,n$ are called \emph{replicas}, and are interpreted as random variables independently drawn from the posterior. When $n=1$ we simply write $f(\vct{x},\vct{x}^*)$ instead of $f(\vct{x}^{(1)},\vct{x}^*)$.  Throughout the rest of the manuscript, we use the following notation: for $l,l'=1,\cdots,n,*$, we let
 \[R_{l,l'} := \vct{x}^{(l)} \cdot \vct{x}^{(l')} = \frac{1}{N} \sum_{i=1}^N x_i^{(l)}x_i^{(l')}.\]
In this section we show the convergence of the first 4 moments of the overlap at optimal rates under some conditions: if either the prior $P_{\x}$ is not symmetric about the origin or the Hamiltonian $H$ is ``perturbed" in the following sense. Let $t \in [0,1]$ and consider the`` interpolating" Hamiltonian (this qualification will become clear in the next section)
\begin{align}\label{interpolating_hamiltonian}
-H_t(\vct{x}) &:=  -\frac{t\lambda}{2N}\sum_{i < j} x_i^2x_j^2 + \sqrt{\frac{t\lambda}{N}}\sum_{i < j} W_{ij}x_ix_j + \frac{t\lambda}{N}\sum_{i < j}x_ix_i^*x_jx_j^* \\
&~~~-\frac{(1-t)r}{2}\sum_{i=1}^N x_i^2 +  \sqrt{(1-t)r}\sum_{i=1}^N z_{i}x_i +  (1-t)r\sum_{i=1}^N x_ix_i^*,\nonumber
\end{align}
where the $z_i$'s are i.i.d.\ standard Gaussian r.v.'s independent of everything else, and $r = \lambda q^*(\lambda)$. We similarly define the Gibbs average $\langle \cdot \rangle_t$ as in~\eqref{gibbs_average_one} where $H$ is replaced by $H_t$. We now state a fundamental property satisfied by both $\langle \cdot \rangle$ and $\langle \cdot \rangle_t$.

\paragraph{The Nishimori property.} The fact that the Gibbs measure $\langle \cdot \rangle$ is a posterior distribution~\eqref{posterior} has far-reaching consequences. A crucial implication is that the $n+1$-tuples $(\vct{x}^{(1)},\cdots,\vct{x}^{(n+1)})$ and $(\vct{x}^{(1)},\cdots,\vct{x}^{(n)},\vct{x}^*)$ have the same law under $\E\langle \cdot \rangle$. This fact, which is a simple consequence of Bayes' rule~\citep[see Proposition 16,][]{lelarge2016fundamental} prevents replica-symmetry from breaking~\citep[see][]{korada2009exact}. In particular, $R_{1,2}$ and $R_{1,*}$ have the same distribution. This bares the name of \emph{the Nishimori property} in the spin glass literature~\citep{nishimori2001statistical}.  
Moreover, this property is preserved under the interpolating Gibbs measure $\langle \cdot \rangle_t$ for all $t \in [0,1]$. Indeed, the interpolation is constructed in such a way that $\langle \cdot \rangle_t$ is the posterior distribution of the signal $\vct{x}^*$ given the augmented set of observations
\begin{align}\label{augmented_observations}
\begin{cases}
Y_{ij} &= \sqrt{\frac{t\lambda}{N}} x^*_ix^*_j + W_{ij}, \quad 1 \le i < j \le N, \\
y_i &= \sqrt{(1-t)r}x^*_i + z_i, \quad 1 \le i \le N, 
\end{cases}
\end{align}  
where one receives side information about $\vct{x}^*$ from a scalar Gaussian channel, $r= \lambda q^*(\lambda)$, and the signal-to-noise ratios of the two channels are altered in a time dependent way. 
Now we state our concentration result.
\begin{theorem}\label{theorem_fourth_moment}
For all $\lambda \in \A$ and all $t \in [0,1]$, there exist constants $K(\lambda) \ge 0$ and $c(t) \ge 0$ such that
\begin{equation}\label{convergence_fourth_moment}
\E\left\langle \left(R_{1,*} - q^*\right)^{4} \right\rangle_t \le K(\lambda)\Big(\frac{1}{N^2} + e^{-c(t)N}\Big).
\end{equation}
Moreover, $c(t) > 0$ on $[0,1)$, and if either $\lambda<\lambda_c$ or $P_{\x}$ is  not symmetric about the origin, then $c(t) \ge c_0$ for some constant $c_0 = c_0(\lambda)>0$. Otherwise, $c(t) \sim c_0(1-t)^2$ as $t \to 1$.
\end{theorem}
 If $P_{\x}$ is symmetric about the origin then the distribution of $R_{1,*}$ under $\E\langle \cdot \rangle$ is also symmetric, so $\E\langle R_{1,*}\rangle = 0$. If moreover $q^*(\lambda)>0$ (i.e., $\lambda > \lambda_c$) then~\eqref{convergence_fourth_moment} becomes trivial at $t=1$ since both sides are constant. On the other hand, if either $t<1$ or $P_{\x}$ is asymmetric, the sign symmetry of the spike is broken. This forces the overlap to be positive and hence concentrate about $q^*(\lambda)$. Finally, if $\lambda < \lambda_c$, $q^*(\lambda) = 0$ and the sign symmetry becomes irrelevant since the overlap converges to zero regardless. 
Let us mention that in the symmetric unperturbed case ($t=1$), we expect a variant of~\eqref{convergence_fourth_moment} to hold where $R_{1,*}$ is replaced by its absolute value in the statement, and the upper bound would be $K/N^2$. Unfortunately, our methods do not allow us to prove such a statement, but we are able to prove a weaker result (see Lemma~\ref{convergence_in_probability}): for all $\epsilon \ge 0$,
\begin{equation}\label{convergence_in_probability_absolute_value}
\E \big\langle \indi\left\{\big| |R_{1,*}| - q^* \big| \ge \epsilon \right\}\big\rangle \longrightarrow 0.
\end{equation}
Although this a minor technical point, we also point out that the estimate $c(t) \sim c_0 (1-t)^2$ in the statement is suboptimal. A heuristic argument allows us to get $c(t) \sim c_0 (1-t)$ as $t \to 1$, but we are currently unable to rigorously justify it. 
 
%A more standard way of perturbing the Hamiltonian is by adding a scalar field $-\frac{r}{2}\sum_{i=1}^N x_i^2 +  \sqrt{r}\sum_{i=1}^N z_{i}x_i +  r\sum_{i=1}^N x_ix_i^*$~\citep{lelarge2016fundamental} (or just $r\sum_{i=1}^N x_ix_i^*$~\citep{panchenko2013sherrington,talagrand2003challenge}) to $H$ (without altering the latter by $t$), then one would send the parameter $r$ to $0$. It is possible to prove a result similar to~\eqref{convergence_fourth_moment} under this form of perturbation, but this does not suit our purposes.   

\paragraph{MMSE.} The bound~\eqref{convergence_fourth_moment} can be used to deduce the optimal error of estimating $\vct{x}^*$ based on the observations~\eqref{augmented_observations}. The posterior mean $\langle \vct{x}\rangle_t$ is the estimator with Minimal Mean Squared Error (MMSE) among all estimators $\widehat{\theta}(\mtx{Y},\vct{y}) \in \R^N$, and the MMSE is 
\begin{align*}
\frac{1}{N} \sum_{i=1}^N \E\left[(x^*_i - \langle x_i \rangle_t)^2\right] &= \E_{P_{\x}}[X^2] - \frac{2}{N} \sum_{i=1}^N\E \langle x_ix_i^*\rangle_t +  \frac{1}{N} \sum_{i=1}^N \E \langle x_i \rangle_t^2\\
&= \E_{P_{\x}}[X^2] - \E \langle R_{1,*} \rangle_t.
\end{align*}
The last line follows from the Nishimori property, since $\E \langle x \rangle_t^2 = \E \langle x^{(1)}x^{(2)} \rangle_t =  \E \langle xx^* \rangle_t$. Theorem~\ref{theorem_fourth_moment} implies in particular (under the conditions of its validity) that $\E \langle R_{1,*} \rangle_t \to q^*(\lambda)$, yielding the value of the MMSE. It is in particular possible to estimate the spike $\vct{x}^*$ from the observations~\eqref{augmented_observations} with non-trivial accuracy if and only if $\lambda > \lambda_c$. Note that at $t=1$ (no side information) the result still holds below $\lambda_c$ or when the prior is not symmetric. Otherwise, as mentioned before, the problem is invariant under a sign flip of $\vct{x}^*$ so one has to change the measure of performance. Beside the result~\eqref{convergence_in_probability_absolute_value}, we are unable to say much in this situation.    

\paragraph{Asymptotic variance.} By Jensen's inequality we deduce from~\eqref{convergence_fourth_moment} the convergence of the second moment: 
\begin{equation}\label{convergence_second_moment}
\E\left\langle \left(R_{1,*} - q^*\right)^{2} \right\rangle_t \le K(\lambda)\Big(\frac{1}{N} + e^{-c(t)N}\Big).
\end{equation}
To establish our finite-size correction result (Theorem~\ref{finite_size_correction}) we need to prove a result stronger than~\eqref{convergence_second_moment}, namely that $N \cdot \E\left\langle \left(R_{1,*} - q^*\right)^{2} \right\rangle_t$ converges to a limit. For $t \in [0,1]$ and $\lambda \in \A$, we let 
\begin{equation}\label{Delta_RS}
\Delta_{\RS}(\lambda;t) := \frac{1}{\lambda} \left(-\frac{\mu_1}{1-t\mu_1} + \frac{2\mu_2}{1-t\mu_2} + \lambda\frac{4a(1) - 3a(2)}{(1-t\mu_1)^2}\right),
\end{equation}
where $\mu_1$ and $\mu_2$ are defined in~\eqref{eigenvalues}.
\begin{theorem}\label{asymptotic_variance}  
For all $\lambda \in \A$ and all $t \in [0,1]$, there exist constants $K(\lambda) \ge 0$ and $c(t) \ge 0$ such that
\[ 
\abs{ N\cdot \E\left\langle \left(R_{1,*}-q^*\right)^2\right\rangle_t - \Delta_{\RS}(\lambda;t)} \le K(\lambda)\left(\frac{1}{\sqrt{N}}+ Ne^{-c(t)N}\right).
\]
Moreover, $c(t) > 0$ on $[0,1)$, and if either $\lambda<\lambda_c$ or $P_{\x}$ is  not symmetric about the origin, then $c(t) \ge c_0$ for some constant $c_0 = c_0(\lambda)>0$. Otherwise, $c(t) \sim c_0(1-t)^2$ as $t \to 1$.
\end{theorem}
The proofs  of Theorems~\ref{theorem_fourth_moment} and~\ref{asymptotic_variance} rely on the cavity method, and will be presented in Section~\ref{sxn:cavity_method}. 
Finally, the techniques we use could be easily extended to prove convergence of all the moments at optimal rates: for all integers $k$,
\[\E\left\langle \left(R_{1,*} - q^*\right)^{2k} \right\rangle_t \le \frac{K(k)}{N^{k}} + K(k)e^{-c(k,t)N},\] 
but we will not need this stronger statement.

%%%%%%%%%%%%%%%%%%%%%%%%%%%%%%%%%
\section{The interpolation method}
\label{sxn:interpolation_method}
In this section we present the interpolation method of~\cite{guerra2001sum}. All our main arguments will rely, in one way or another, on this method. Along the way, we prove Theorem~\ref{finite_size_correction}. The idea is to construct a continuous interpolation path between the Hamiltonian $H$ and a simpler Hamiltonian that decouples all the variables, and analyze the incremental change in the free energy along the path. We present two versions of this method. The first one is the classical method which is applied to the free energy of the entire system, and a second one applied to the free energy of a more restricted system.  

\subsection{The Guerra interpolation}
Our interpolating Hamiltonian is $H_t$ from~\eqref{interpolating_hamiltonian} with $r=\lambda q$ for some $q \ge 0$.
Now we consider the interpolating free energy
\begin{equation}\label{interpolating_free_energy}
\varphi(t) :=  \frac{1}{N} \E \log \int e^{- H_t(\vct{x})} \rmd P_{\x}^{\otimes N}(\vct{x}).
\end{equation}
We see that $\varphi(1) = F_{N}$ and $\varphi(0) = \psi(\lambda q)$. This function is moreover differentiable in $t$, and 
by differentiation, we have
\begin{align*}
\varphi'(t) &= \frac{1}{N} \E \left \langle -\frac{\rmd H_t(\vct{x})}{\rmd t}\right\rangle_t\\
&= \frac{1}{N} \E \left\langle -\frac{\lambda}{2N}\sum_{i < j} x_i^2x_j^2 + \half \sqrt{\frac{\lambda}{t N}}\sum_{i < j} W_{ij}x_ix_j + \frac{\lambda}{N}\sum_{i < j}x_ix_i^*x_jx_j^* \right\rangle_t \\
&~~~+\frac{1}{N}\E \left\langle \frac{\lambda q}{2}\sum_{i=1}^N x_i^2 -  \half\sqrt{\frac{\lambda q}{1-t}}\sum_{i=1}^N z_{i}x_i - \lambda q\sum_{i=1}^N x_ix_i^*\right\rangle_t.
\end{align*}
Now we use Gaussian integration by parts to eliminate the variables $W_{ij}$ and $z_i$. The details of this computation are explained extensively in many sources. See~\citep{talagrand2011mean1,krzakala2016mutual,lelarge2016fundamental}. We get
\begin{align*}
\varphi'(t) &= -\frac{\lambda}{2N^2} \E\left\langle \sum_{i < j} x_i^{(1)} x_j^{(1)}x_i^{(2)} x_j^{(2)}\right\rangle_t + \frac{\lambda}{N^2}  \E\left\langle  \sum_{i < j}  x_i x_i^{*} x_j x_j^{*} \right\rangle_t \\
&~~~+\frac{\lambda q}{2N} \E \left\langle \sum_{i=1}^N x_i^{(1)} x_i^{(2)} \right\rangle_t - \frac{\lambda q}{N} \E \left\langle \sum_{i=1}^N x_ix_i^*\right\rangle_t.
\end{align*}
Completing the squares yields
\begin{align}\label{derivative_before_nishimori}
\varphi'(t) &= -\frac{\lambda}{4} \E\left\langle (\vct{x}^{(1)}\cdot\vct{x}^{(2)}-q)^2\right\rangle_t + \frac{\lambda}{4} q^2 + \frac{\lambda}{4N^2} \sum_{i=1}^N \E\left\langle {x_i^{(1)}}^2{x_i^{(2)}}^2\right\rangle_t \\
&~~~+\frac{\lambda}{2} \E\left\langle (\vct{x}\cdot\vct{x}^*-q)^2\right\rangle_t - \frac{\lambda}{2} q^2 - \frac{\lambda}{2N^2} \sum_{i=1}^N \E\left\langle {x_i}^2{x_i^{*}}^2\right\rangle_t.\nonumber
\end{align}
The first line in the above expression involves overlaps between two independent replicas, while the second one involves overlaps between one replica and the planted solution. Using the Nishimori property, the derivative of $\varphi$ can be written as
\begin{equation}\label{free_energy_derivative}
\varphi'(t) = \frac{\lambda}{4} \E\left\langle (R_{1,*}-q)^2 \right\rangle_t -\frac{\lambda}{4} q^2 - \frac{\lambda}{4N} \E\left\langle {x_N}^2{x_N^{*}}^2\right\rangle_t.
\end{equation}  
The last term follows by symmetry between sites. Now, integrating over $t$, the difference between the free energy and the $\RS$ potential $F(\lambda,q)$ can be written in the form of a sum rule:  
\begin{equation}\label{sum_rule}
F_N - F(\lambda,q) = \frac{\lambda}{4} \int_{0}^1 \Big(\E\left\langle (R_{1,*}-q)^2\right\rangle_t - \frac{1}{N}\E\left\langle {x_N}^2 {x_N^*}^2\right\rangle_t\Big) \rmd t.
\end{equation}
 
We see from~\eqref{sum_rule} that $F_N$ converges to $F(\lambda,q)$ if and only if the overlap $R_{1,*}$ concentrates about $q$. This happens only for a value of $q$ that maximizes the $\RS$ potential $F(\lambda, \cdot)$. Using Theorem~\ref{theorem_fourth_moment} one can already prove the $1/N$ optimal rate below $\lambda_c$ or above it when the prior is not symmetric. Indeed since $c(t)$ is lower-bounded by a positive constant in this case, the bound~\eqref{convergence_second_moment} yields $\int_0^1\E \langle (R_{1,*} - q^*)^2 \rangle_t \rmd t \le K(\lambda)/N$. Also, the second integrand in~\eqref{sum_rule} is bounded by $K/N$ for some constant $K \ge 0$, so we have for all $\lambda \in \A$, $F_N = \phi_{\RS}(\lambda) + \bigo(1/N)$. If $\lambda > \lambda_c$ and the prior is symmetric then we are only able to prove a rate of $1/\sqrt{N}$ due to the fact $c(t) \sim c_0(1-t)^2$ as $t \to 1$. The $1/N$ rate would follow immediately in this case if one is able to improve the latter estimate to $c(t) \sim c_0(1-t)$.
To go further, we use Theorem~\ref{asymptotic_variance}, and the additional fact that $\E\langle {x_N}^2 {x_N^*}^2\rangle_t$ has a limit:
\begin{lemma}\label{diagonal_term}  
For all $\lambda \in \A$ and for all $t \in [0,1)$, there exist constants $K(\lambda) \ge 0$ and $c(t) \ge 0$ such that
\[\abs{\E\left\langle {x_N}^2 {x_N^*}^2\right\rangle_t - a(0)} \le K(\lambda)\left(\frac{1}{\sqrt{N}}+ e^{-c(t)N}\right).\]
Moreover, $c(t) > 0$ on $[0,1)$, and if either $\lambda<\lambda_c$ or $P_{\x}$ is  not symmetric about the origin, then $c(t) \ge c_0$ for some constant $c_0 = c_0(\lambda)>0$. Otherwise, $c(t) \sim c_0(1-t)^2$ as $t \to 1$.
\end{lemma}
The proof of Lemma~\ref{diagonal_term} relies on the cavity method, and will be presented in the Section~\ref{sxn:cavity_method}. Now we are ready to prove Theorem~\ref{finite_size_correction}.

\vspace{.2cm}

%%%%%%%%%%%%%%%%%%%%%%%%%%%%%%
\begin{proofof}{Theorem~\ref{finite_size_correction}}
By formula~\eqref{sum_rule} with the choice $q = q^*(\lambda)$, we have
\begin{align*}
\abs{N(F_N - \phi_{\RS}(\lambda)) - \frac{\lambda}{4}\Big(\int_{0}^1 \Delta_{\RS}(\lambda;t)\rmd t - a(0)\Big)} 
&\le \frac{\lambda}{4} \int_0^1 \abs{ N \E\left\langle \left(R_{1,*}-q^*\right)^2\right\rangle_t - \Delta_{\RS}(\lambda;t)} \rmd t \\
&~+  \frac{\lambda}{4} \int_0^1 \abs{\E\left\langle {x_N}^2 {x_N^*}^2\right\rangle_t - a(0)} \rmd t.
\end{align*}
By Theorem~\ref{asymptotic_variance} and Lemma~\ref{diagonal_term}, the integrands on the right-hand side are bounded by $K/\sqrt{N} + KNe^{-c(t)N}$ where $c(t) > c_0 >0$ for all $t$ in the cases $\lambda < \lambda_c$ or $P_{\x}$ not symmetric about the origin, so the convergence follows. The function $\psi_{\RS}(\lambda)$ is the second term in the left-hand side. Formula~\eqref{sigma_RS} follows by integration.
\end{proofof}

%%%%%%%%%%%%%%%%%%%%%%%%%%%%%%
\subsection{The main estimate: energy gap at suboptimal overlap}
\label{sxn:main_estimate}
Recall the interpolating Hamiltonian $H_t$ from~\eqref{interpolating_hamiltonian} with $r = \lambda q^*(\lambda)$. Let us now introduce the \emph{Franz-Parisi potential}~\citep{franz1995recipes}. For $m \in \R$ and $\epsilon>0$ we define 
\begin{equation}\label{free_energy_fixed_overlap}
\Phi_{\epsilon}(m;t) := \frac{1}{N}\E\log \int \indi\{R_{1,*} \in [m,m+\epsilon)\} e^{-H_t(\vct{x})}  \rmd P_{\x}^{\otimes N}(\vct{x}).
\end{equation}
This is the free energy of a subsystem of configurations having an overlap close to a fixed value $m$ with the planted signal $\vct{x}^*$. It is clear that $\Phi_{\epsilon}(m;t) \le \varphi(t)$, where the latter is the interpolating free energy defined in~\eqref{interpolating_free_energy}.  
The purpose of this section is to prove that when $m$ is far from $q^*(\lambda)$ then there is a sizable gap between $\Phi_{\epsilon}(m;t)$ and $\varphi(t)$. This estimate is a main ingredient is our proof of overlap concentration. (The other main ingredient is the cavity method, which will be presented in the next section.) To prove this we will need the auxiliary function
\[\phi_{\RS}(\lambda; t) = \sup_{q \ge 0} ~\psi(\lambda q) - \frac{t\lambda q^2}{4}.\]
One can show that the above formula is the limit of $\varphi(t)$ as $N \to \infty$, for example, by using the so called ``Aizenman-Sims-Starr scheme"~\citep[]{aizenman2003extended}; see~\citep{lelarge2016fundamental}. For our purposes we will only need the inequality
\begin{equation}\label{lower_bound_interpolated}
\varphi(t) \ge \phi_{\RS}(\lambda; t) - \frac{K\lambda t}{N},
\end{equation}
which can be proved using the interpolation method presented in the previous section and dropping the non-negative term $\E\langle (R_{1,*} - q^*)^2 \rangle$ from the expression analogous to~\eqref{free_energy_derivative} in this case. Now it suffices to compare $\Phi_{\epsilon}(m;t)$ to $\phi_{\RS}(\lambda;t)$. The result is given in Proposition~\ref{energy_deficit}, and we finish this subsection by Proposition~\ref{convergence_in_probability} showing convergence in probability of the overlaps as a straightforward consequence.

%%%%%%%%%%%%%%%%%%%%%%%%%%%%%%

For $r \ge 0$ and $s\in \R$, we let 
\begin{equation}\label{psi_bar}
\widebar{\psi}(r,s) := \E_{x^*,z} \log \int \exp\left(\sqrt{r}zx + s xx^* - \frac{r}{2} x^2\right) \rmd P_{\x}(x).
\end{equation}
We see that $\widebar{\psi}(r,r) = \psi(r)$, but unlike $\psi$, the function $\widebar{\psi}$ does not have an interpretation as the $\KL$ between two distributions. The next lemma states some key properties of this function that will be useful later on.

\begin{lemma} \label{asymmetry_lower_bound}
For all $r\ge 0$, it holds that
\begin{itemize}
\item The function $s \mapsto \widebar{\psi}(r,s)$ is strictly convex, hence strongly convex on any compact. 
\item There exist a constant $c = c(r,P_{\x}) \ge 0$ such that $\widebar{\psi}(r,-r) \le \widebar{\psi}(r,r) - c$. If $r >0$ then $c >0$ \emph{unless} the prior $P_{\x}$ is symmetric about the origin (in which case $\widebar{\psi}(r,-r) = \widebar{\psi}(r,r)$).   
\item The map $r \mapsto c(r,P_{\x})$ is increasing on $\R_+$. 
\end{itemize}
\end{lemma}

The proof of the above lemma can be found in the Appendix. We now state a useful interpolation bound on $\Phi_{\epsilon}(m;t)$. This is a simpler version of the Guerra-Talagrand $\oneRSB$ interpolation bound at fixed overlap, a key invention that ultimately paved the way towards a proof of the Parisi formula~\citep{guerra2003broken,talagrand2006parisi}. In some sense, since we are dealing with a planted model, we only need a replica-symmetric version of this bound.    
\begin{proposition} \label{fixed_overlap_upper_bound}
Fix $m \in \R$, $\epsilon>0$, $t\in [0,1]$ and $\lambda \ge 0$. Let $r = (1-t) \lambda q^* + t \lambda |m|$, $\bar{r} = (1-t) \lambda q^* + t \lambda m$. 
There exist constants $K = K(P_{\x}) > 0$, $K' = K'(\lambda) >0$ such that
\begin{align*}
\Phi_{\epsilon}(m; t) &\le  \widebar{\psi}\big(r, \bar{r}\big) - \frac{t\lambda m^2}{4}  - K\left(m - \partial_s\widebar{\psi}\big(r, \bar{r}\big)\right)^2 
+ K' \epsilon^2 + \frac{K'}{N}.
\end{align*}
\end{proposition}
\begin{proof}
To obtain a bound on $\Phi_{\epsilon}(m;t)$ for any fixed $t$, we use the interpolation method with Hamiltonian 
\begin{align*}
-H_{t,s}(\vct{x}) &:=  \sum_{i < j}-\frac{ts\lambda}{2N} x_i^2x_j^2 + \sqrt{\frac{ts\lambda}{N}} W_{ij}x_ix_j + \frac{ts\lambda}{N}x_ix_i^*x_jx_j^*\\
&~+\sum_{i=1}^N -\frac{(1-t)\lambda q^*}{2} x_i^2 +  \sqrt{(1-t)\lambda q^*} z_{i}x_i +  (1-t)\lambda q^* x_ix_i^*\\
&~+\sum_{i=1}^N -\frac{(1-s)t\lambda |m|}{2} x_i^2 +  \sqrt{(1-s)t\lambda |m|} z_{i}'x_i +  (1-s)t\lambda m x_ix_i^*,
\end{align*}
by varying $s\in[0,1]$. The r.v.'s $W, z, z'$ are all i.i.d.\ standard Gaussians independent of everything else. We define
\[\varphi(t,s) :=  \frac{1}{N} \E \log \int  \indi\{R_{1,*} \in [m,m+\epsilon)\} e^{- H_{t,s}(\vct{x})} \rmd P_{\x}^{\otimes N}(\vct{x}).\]
We compute the derivative w.r.t.\ $s$. The same algebraic manipulations conducted in the computation of $\varphi'$ up to~\eqref{derivative_before_nishimori} apply here, and we get
\begin{align*}
\partial_s \varphi(t,s) = & -\frac{\lambda t}{4} \E\left\langle (\vct{x}^{(1)}\cdot\vct{x}^{(2)}-|m|)^2\right\rangle_{t,s} + \frac{\lambda t}{4} |m|^2 + \frac{\lambda t}{4N^2} \sum_{i=1}^N \E\left\langle {x_i^{(1)}}^2{x_i^{(2)}}^2\right\rangle_{t,s} \\
&~~~+\frac{\lambda t}{2} \E\left\langle (\vct{x}\cdot\vct{x}^*- m)^2\right\rangle_{t,s} - \frac{\lambda t}{2} m^2 - \frac{\lambda t}{2N^2} \sum_{i=1}^N \E\left\langle {x_i}^2{x_i^{*}}^2\right\rangle_{t,s},
\end{align*}
where $\langle \cdot \rangle_{t,s}$ is the Gibbs average w.r.t.\ the Hamiltonian $-H_{t,s}(\vct{x}) +\log \indi\{\vct{x}\cdot\vct{x}^* \in [m,m+\epsilon)\}$. A few things now happen. Notice that the planted term (first term in the second line) is trivially smaller than $t\lambda \epsilon^2/2$ due to the overlap restriction. Moreover, the last terms in both lines are of order $1/N$ since the variables $x_i$ are bounded. The first term in the first line, which involves the overlap between two replicas, is more challenging. What makes this term difficult to control is that the Gibbs measure $\langle \cdot \rangle_{t,s}$ no longer satisfies the Nishimori property due to the overlap restriction, so the overlap between two replicas no longer has the same distribution as the overlap of one replica with the planted spike. Fortunately, this term is always non-positive so we can ignore it altogether and obtain an upper bound:
\begin{equation*}
\partial_s \varphi(t,s) \le - \frac{\lambda t}{4} m^2 +\frac{\lambda t \epsilon^2}{2}+ \frac{\lambda  K}{N}.
\end{equation*}   
Integrating over $s$, we get
\[\Phi_{\epsilon}(m;t) \le \varphi(t,0) - \frac{\lambda t}{4} m^2 +\frac{\lambda t \epsilon^2}{2} + \frac{\lambda K}{N}.\] 
Now it remains to show that  
\[\varphi(t,0) \le \widebar{\psi}(r,\bar{r}) - K\left(m - \partial_s\widebar{\psi}(r,\bar{r})\right)^2 + \bigo(\epsilon^2).\]
By properties of the Gaussian distribution, we can write $-H_{t,0} = \sum_{i=1}^N  \sqrt{r} z_i x_i +  \bar{r} x_ix_i^* -\frac{r}{2} x_i^2$.
Define the following (random) probability measure
\[G(A) := \frac{\int_A e^{-H_{t,0}(\vct{x})}  \rmd P_{\x}^{\otimes N}(\vct{x})}{\int e^{-H_{t,0}(\vct{x})}  \rmd P_{\x}^{\otimes N}(\vct{x})},\]
for all Borel sets $A \subseteq \R^N$. We observe that conditionally on the Gaussian vector $\vct{z}$ and the planted vector $\vct{x^*}$, $G$ is a product measure due to the additive form of $H_{t,0}$. Moreover,
\[\varphi(t;0)  -  \widebar{\psi}(r,\bar{r}) = \frac{1}{N} \E\log G(\{R_{1,*} \in [m,m+\epsilon)\}),\]
so will be interested in a large deviation bound on the above quantity.
The prior $P_{\x}$ is of bounded support, thus the marginals of $G$ (conditional on $\vct{x}^*$ and $\vct{z}$) are clearly sub-Gaussian. Therefore, by concentration of measure, the empirical average $\vct{x}\cdot\vct{x}^*/N$ must concentrate around its expectation: for all $u\ge 0$
\[G\left(\left\{\frac{1}{N}\sum_{i=1}^N (x_i -\E_{G}[x_i])x^*_i \ge u\right\}\right) \le e^{-N u^2/2K^2},\]
where $K$ is for instance twice the diameter of the support of $P_{\x}$. 
This implies
\[\frac{1}{N} \log G(\{R_{1,*} \in [m,m+\epsilon)\}) \le -\frac{(m - \widehat{q})^2}{2K^2}\indi\{\widehat{q} \le m\} 
-\frac{(m +\epsilon - \widehat{q})^2}{2K^2}\indi\{\widehat{q} \ge m+\epsilon\},\]
where $\widehat{q} = \frac{1}{N}\sum_{i=1}^N \E_{G}[x_i]x^*_i$. 
Now by Jensen's inequality ($x \mapsto (x-a)^2\indi\{x \le a\}$ and $x \mapsto (x-b)^2\indi\{x \ge b\}$ are convex), we can write
\begin{align*}
\frac{1}{N}\E \log G(\{R_{1,*} \in [m,m+\epsilon)\}) &\le -\frac{(m -\E[\widehat{q}])^2}{2K^2}\indi\left\{\E[\widehat{q}] \le m\right\} \\
&~~- \frac{(m +\epsilon - \E[\widehat{q}])^2}{2K^2}\indi\left\{\E[\widehat{q}] \ge m+\epsilon\right\}.
\end{align*}
Since
\[\E_G[x_i] = \frac{\int x e^{\sqrt{r}z_i x + \bar{r}x x_i^* - \frac{r}{2}x^2}  \rmd P_{\x}(x)}{\int e^{\sqrt{r}z_i x +\bar{r}x x_i^* - \frac{r}{2}x^2}  \rmd P_{\x}(x)},\] 
we have $\E[\widehat{q}] = \partial_s\psi(r,\bar{r})$. We now use the elementary inequality $\half (x-a)^2 - (a-b)^2 \le (x-a)^2\indi\{x \le a\} +(x-b)^2\indi\{x \ge b\}$, valid for all $x \in \R$ and $a \le b$, to simplify the above bound and obtain 
\[\frac{1}{N}\E \log G(\{R_{1,*} \in [m,m+\epsilon)\}) \le -\frac{(m - \partial_s\widebar{\psi}(r,\bar{r}))^2}{4K^2} + \frac{\epsilon^2}{2K^2}.\]
This allows us to conclude.
\end{proof}

%\vspace{.5cm}
%%%%%%%%%%%%%%%%%%%%%%%%%%%%%%
  
A consequence of the above proposition is an energy gap property: if $m$ is far from $q^*(\lambda)$ then the free energy $\Phi_{\epsilon}(m;t)$ of the configurations having overlap $m$ with $\vct{x}^*$ is strictly smaller than $\phi_{\RS}(\lambda;t)$:  
\begin{proposition}\label{energy_deficit}
For all $\lambda \in \A$, all $\epsilon > 0$ and all $t \in [0,1]$, there exist constants  $c = c(\lambda,\epsilon, t, P_{\x}) \ge 0$ and $\epsilon' = \epsilon'(\lambda,\epsilon)>0$ such that
\[\forall m \in \R \qquad |m - q^*(\lambda)| \ge \epsilon \qquad \implies \qquad \Phi_{\epsilon'}(m;t) \le \phi_{\RS}(\lambda;t) - c.\]
Moreover, if $t <1$ then $c>0$. If either $\lambda <\lambda_c$ or $P_{\x}$ is not symmetric about the origin $\inf_{t \in [0,1]} c(t)>0$. Lastly, if $\lambda > \lambda_c$ and $P_{\x}$ is symmetric, then $c(t) \sim c_0(1-t)$ as $t \to 1$, for some $c_ 0 = c_0(\lambda,\epsilon,P_{\x})>0$.
\end{proposition}

A direct consequence of the above energy gap result is the convergence in probability of the overlaps: 
\begin{proposition}\label{convergence_in_probability}
For all $\lambda \in \A$, all $\epsilon >0$ and all $t \in [0,1]$, there exist constants $K = K(\lambda,\epsilon) \ge 0,~ c = c(\lambda,\epsilon, t,P_{\x}) \ge 0$ such that 
\[\E \left\langle \indi\left\{\big|R_{1,*} - q^*(\lambda)\big| \ge \epsilon \right\}\right\rangle_t \le K e^{-cN},\]
where the constant $c$ the same properties as in Proposition~\ref{energy_deficit}, except that $c(t) \sim c_0(1-t)^2$ as $t \to 1$.   
Moreover, if $\lambda > \lambda_c$, $P_{\x}$ is symmetric and $t=1$ then one still has
\[\E \left\langle \indi\left\{\big||R_{1,*}| - q^*(\lambda)\big| \ge \epsilon \right\}\right\rangle_t \le K e^{-cN},\]
with $c = c(\lambda,\epsilon, P_{\x})>0$.
\end{proposition}

To prove the above lemma, we first show that the partition function of the model enjoys sub-Gaussian concentration in logarithmic scale. This is an elementary consequence of two classical concentration-of-measure results: concentration of Lipschitz functions of Gaussian random variables, and concentration of \emph{convex} Lipschitz function of \emph{bounded} random variables. See~\cite{boucheron2013concentration} and~\cite{van2014probability}.   

%%%%%%%%%%%%%%%%%%%%%%%%%%%%%%
%\vspace{.5cm}
\begin{lemma} \label{sub_gaussian_concentration}
Let $A$ be a Borel subset of $\R^N$, and define the random variable
\[Z := \int_A e^{-H_t(\vct{x})}  \rmd P_{\x}^{\otimes N}(\vct{x}).\] 
There exist a constant $K>0$ depending only on $\lambda$ and $P_{\x}$ such that for all $u \ge 0$, 
\[\Pr\left(\abs{\frac{1}{N} \log Z - \frac{1}{N} \E\log Z } \ge u\right) \le 4e^{-Nu^2/K}.\]
\end{lemma}

\vspace{.5cm}
\begin{proofof}{Lemma~\ref{sub_gaussian_concentration}} It suffices to notice that the map $(\mtx{W},\vct{z}) \mapsto \frac{1}{N}\log Z(\mtx{W},\vct{z})$ is Lipschitz with constant $K\sqrt{\frac{\lambda}{N}}$ for every $\vct{x}^* \in \R^N$, and that the map $\vct{x}^* \mapsto \frac{1}{N}\E[\log Z |\vct{x^*}]$ (where the expectation is over $\mtx{W}, \vct{z}$) is convex and Lipschitz with constant $K \frac{\lambda}{\sqrt{N}}$. Then we use concentration of Lipschitz functions of Gaussian r.v.'s and of convex Lipschitz functions of bounded r.v.'s (since the coordinates $x^*_i$ are bounded).   
\end{proofof}

\vspace{.5cm}
\begin{proofof}{Proposition~\ref{convergence_in_probability}} 
For $\epsilon,\epsilon'>0,~ t\in [0,1)$, we can write the decomposition
\begin{align*}
\E\left\langle \indi\left\{ |R_{1,*}-q^*(\lambda)| \ge \epsilon \right\} \right\rangle_t  &= \sum_{l \ge 0} \E\left\langle \indi \big\{ R_{1,*}-q^*-\epsilon \in [l\epsilon',(l+1)\epsilon') \big\}\right\rangle_t\\
& ~+\sum_{l \ge 0} \E\left\langle \indi \big\{ -R_{1,*}+q^*-\epsilon \in [l\epsilon',(l+1)\epsilon') \big\}\right\rangle_t,
\end{align*}
where the integer index $l$ ranges over a finite set of size $ \le K/\epsilon'$ since the prior $P_{\x}$ has bounded support. We will only treat the first sum in the above expression since the argument extends trivially to the second sum. Let $A = \big\{ R_{1,*}-q^*-\epsilon \in [l\epsilon',(l+1)\epsilon') \big\}$ and write
\begin{equation}\label{indicator_decomposition_2}
\E\left\langle \indi(A)\right\rangle_t = \E \left[\frac{\int_A e^{-H_t(\vct{x})}  \rmd P_{\x}^{\otimes N}(\vct{x}) }{\int e^{-H_t(\vct{x})}  \rmd P_{\x}^{\otimes N}(\vct{x})}\right].
\end{equation}
In virtue of Lemma~\ref{sub_gaussian_concentration} the two quantities in the above fraction enjoy sub-Gaussian concentration in logarithmic scale. For any given $l$ and $u \ge 0$, we simultaneously have
\begin{align*}
\frac{1}{N}\log\int e^{-H_t(\vct{x})}  \rmd P_{\x}^{\otimes N}(\vct{x}) &\ge \frac{1}{N}\E \log\int e^{-H_t(\vct{x})}  \rmd P_{\x}^{\otimes N}(\vct{x}) - u\\
&\ge  \phi_{\RS}(\lambda;t) - \frac{K t \lambda}{N} - u,
\end{align*}
(the last inequality come from~\eqref{lower_bound_interpolated}), and
\begin{align*}
\frac{1}{N}\log \int_A e^{-H_t(\vct{x})}  \rmd P_{\x}^{\otimes N}(\vct{x}) 
&\le \frac{1}{N}\E\log \int_A e^{-H_t(\vct{x})}  \rmd P_{\x}^{\otimes N}(\vct{x}) + u\\
&= \Phi_{\epsilon'}(q^* + \epsilon + l\epsilon' ;t)+u,
\end{align*}
with probability at least $1-4e^{-Nu^2/K}$. On the complement of this event event, we simply bound the fraction in~\eqref{indicator_decomposition_2} by 1. Combining the above bounds we have
\[\E\left\langle \indi(A) \right\rangle_t \le 4e^{-Nu^2/K} + e^{N(\delta+2u)},\]
where $\delta = \Phi_{\epsilon'}(m;t) -  \phi_{\RS}(\lambda;t) + K\lambda t/N$ with $m =q^* + \epsilon+ l\epsilon'$. We let $\epsilon'$ be a function of $\lambda$ and $\epsilon$ as dictated by Proposition~\ref{energy_deficit}, and $c>0$ such that $\delta \le -c$ for all $m$ such that $|m-q^*| \ge \epsilon$. Now we conclude by letting $u = -\delta/3$. 
Finally, if $P_{\x}$ is symmetric and $t=1$, then it suffices to consider non-negative values of $m$ in the above argument to prove the corresponding statement.  
\end{proofof}

\vspace{.5cm}
\begin{proofof}{Propopsition~\ref{energy_deficit}}
The gap we seek to prove will come from different sources, depending on the particular cases we will look at. The treatment will be split into several nested cases, depending on whether $t$ is small or not and whether $m$ positive and negative. 
\paragraph{Large $t$.} Assume $t \ge t_0$ to be determined later.  
For $m \ge 0$, Proposition~\ref{fixed_overlap_upper_bound} implies
\[\Phi_{\epsilon'}(m;t) \le \psi((1-t)\lambda q^*+t\lambda m) - \frac{t\lambda m^2}{4} +K'\epsilon'^2 + \frac{K'}{N}.\] 
Since $\psi$ is a convex function we have
\begin{align}\label{interpolation_positive_overlap}
\psi((1-t)\lambda q^*+t\lambda m) - \frac{t\lambda m^2}{4} &\le (1-t)\psi(\lambda q^*)+ t\psi(\lambda m) - \frac{t\lambda m^2}{4} \nonumber\\
&= (1-t)\psi(\lambda q^*) + tF(\lambda,m).
\end{align} 
Since $q^*(\lambda)$ is the unique maximizer of $m \mapsto F(\lambda, m)$, $|m-q^*| \ge \epsilon >0$ implies that $F(\lambda,m) \le F(\lambda,q^*) - c(\epsilon)$ for some $c(\epsilon)>0$. This in turn implies 
\begin{align*}
\Phi_{\epsilon'}(m;t) &\le (1-t)\psi(\lambda q^*) + t\big(\psi(\lambda q^*) - \frac{\lambda q^{*2}}{4} - c(\epsilon)\big) +K'\epsilon'^2 + \frac{K'}{N} \\
&= \psi(\lambda q^*) -  \frac{t\lambda q^{*2}}{4} - tc(\epsilon) +K'\epsilon'^2 + \frac{K'}{N}\\
&\le \phi_{\RS}(\lambda;t) - \frac{t_0c(\epsilon)}{2} + \frac{K'}{N},
\end{align*} 
where  we have chosen $\epsilon'$ such that $K' \epsilon'^2 < t_0c(\epsilon)/2$.
The conclusion is reached for $m \ge 0$. Now we would like to prove the same bound on $\Phi_{\epsilon'}$ for negative overlaps. 
Proposition~\ref{fixed_overlap_upper_bound} implies that for $m >0$,
\begin{equation}\label{interpolation_negative_overlap}
\Phi_{\epsilon'}(-m;t) \le \widebar{\psi}((1-t)\lambda q^*+t\lambda m, (1-t)\lambda q^*-t\lambda m) - \frac{t\lambda m^2}{4} +K' \epsilon'^2+ \frac{K'}{N}.
\end{equation}

If $\lambda < \lambda_c$ then $q^*(\lambda)=0$, and by Proposition~\ref{asymmetry_lower_bound} we have
$\widebar{\psi}(t\lambda m, -t\lambda m) - \frac{t\lambda m^2}{4} \le \psi(t \lambda m) - \frac{t\lambda m^2}{4} \le t_0 F(\lambda,m)$,
and we finish the argument as in the case of positive overlap.
Now we deal with the case $\lambda > \lambda_c$.

\begin{itemize}
\item Suppose $|m - q^*| \ge \epsilon$. We let $r = (1-t)\lambda q^*+t\lambda m$ and $\alpha = \frac{(1-t)q^*}{(1-t)q^*+tm}$.
\begin{align*}
\widebar{\psi}((1-t)\lambda q^*+t\lambda m, (1-t)\lambda q^* -t\lambda m) &= \widebar{\psi}(r, \alpha r - (1-\alpha)r)\\ 
&\le \alpha\widebar{\psi}(r, r) +(1-\alpha)\widebar{\psi}(r, -r) \\
&\le \psi(r)
\end{align*}
where the last two lines follow from Proposition~\ref{asymmetry_lower_bound}. Since $|m - q^*| \ge \epsilon$ we finish once again as in the positive overlap case, starting from line~\eqref{interpolation_positive_overlap}.
\item Suppose $|m - q^*| \le \epsilon$. Then $1-\alpha \ge t_0\frac{q^*-\epsilon}{q^*+\epsilon}$ and $|r - \lambda q^*| \le \lambda \epsilon$.  If $P_{\x}$ is asymmetric we use the bounds of Proposition~\ref{asymmetry_lower_bound}:
\begin{align*}
\widebar{\psi}((1-t)\lambda q^*+t\lambda m, (1-t)\lambda q^* -t\lambda m) &\le \alpha\widebar{\psi}(r, r) +(1-\alpha)\widebar{\psi}(r, -r) \\
&\le \psi(r) - (1-\alpha)c(r)\\
&\le \psi(r) - t_0\frac{q^*-\epsilon}{q^*+\epsilon} c(\lambda (q^*-\epsilon)).
\end{align*}
The last line follows since $r \mapsto c(r)$ is increasing. Then we finish the argument by plugging this bound in~\eqref{interpolation_negative_overlap}. 
\end{itemize}

 \paragraph{Small $t$.} Assume now that $t \le t_0$. In this situation, we draw the gap from the term $(m - \partial_s\bar{\psi}(r,\bar{r}))^2$ (so far unused) in Proposition~\ref{fixed_overlap_upper_bound}.
 The functions $\bar{\psi}(r,\cdot)$ and $\bar{\psi}(\cdot,\cdot) = \psi(\cdot)$ have bounded second derivatives so 
 \[\max~ \big\{\abs{\partial_s\bar{\psi}(r,\bar{r}) - \partial_s\bar{\psi}(r,r)} ~~, ~~ \abs{\partial_s\bar{\psi}(r,r) - \partial_s\bar{\psi}(q^*,q^*)} \big\}~~\le K\lambda t_0.\]  
 Moreover,
\[(m - q^*)^2 \le 2 (m - \partial_s\bar{\psi}(r,\bar{r}))^2 + 2(q^* - \partial_s\bar{\psi}(r,\bar{r}))^2.\]
Since $\partial_s\bar{\psi}(q^*,q^*) = q^*$ we have
\[ (m - \partial_s\bar{\psi}(r,\bar{r}))^2 \ge \half (m - q^*)^2 - K\lambda^2 t_0^2 \ge \frac{\epsilon^2}{2}  - K\lambda^2 t_0^2,\]
and here we choose $t_0$ to be accordingly small, and we finish the argument.  

Note that the assumption that $P_{\x}$ is not symmetric about the origin is used only in the case where the (negative) overlap $-m$ is close to $-q^*$. Consequently, the gap is independent of $t$ in all cases. Alternatively, without this asymmetry assumption (and when $q^*>0$), we see that there is no hope of a gap independent of $t$ since the potential $\Phi_{\epsilon'}(m;t)$ is closer and closer to being even as $t \to 1$. But we can still obtain a gap that depends on $t(1-t)$ via a strong convexity argument. 

The $s \mapsto \widebar{\psi}(r,s)$ is strongly convex on any interval, and for all $r\ge 0$. Therefore, recalling $r = (1-t)\lambda q^*+t\lambda m$ and $\alpha = \frac{(1-t)q^*}{(1-t)q^*+tm}$, there exists a constant $c>0$ depending only on $\lambda$ and $P_{\x}$ (this constant is a bound on $r$) such that
\begin{align*}
\widebar{\psi}((1-t)\lambda q^*+t\lambda m, (1-t)\lambda q^* -t\lambda m) &= \widebar{\psi}(r, \alpha r - (1-\alpha)r)\\ 
&\le \alpha\widebar{\psi}(r, r) +(1-\alpha)\widebar{\psi}(r, -r) - \frac{c}{2} \alpha(1-\alpha) (2r)^2\\
&=  \alpha\widebar{\psi}(r, r) +(1-\alpha)\widebar{\psi}(r, -r) - 2 c t(1-t) \lambda^2 q^* m\\
&\le \widebar{\psi}(r, r) - 2 c t(1-t) \lambda^2 q^*(q^*-\epsilon),
\end{align*}
where the last bounds follows from $\widebar{\psi}(r, -r) \le \widebar{\psi}(r, r)$ and $|m-q^*| \le \epsilon$ (recall that this is the only case where such an argument is needed). 
\end{proofof}

%%%%%%%%%%%%%%%%%%%%%%%%%%%%%%
\section{The cavity method}
\label{sxn:cavity_method}
Now that we have established the convergence in probability of $R_{1,*}$ to $q^*(\lambda)$ under $\E\langle \cdot \rangle_t$ in Lemma~\ref{convergence_in_probability}, we use the cavity method to prove the convergence of the moments of the overlap. 
In its essence, the cavity method amounts to isolating one variable from the system and analyzing the influence of the rest of the variables on it. It was initially introduced as an analytic tool, alternative to the replica method, to solve certain models of spin glasses~\citep{mezard1990spin}, and has since been tremendously successful in predicting the behavior of many mean-field models.    
The underlying principle is know as the \emph{leave-one-out} method in statistics. In our setting, this principle is materialized in the form of an interpolation method (yet again) that separates the last variable from the rest. 

Our proofs of Theorems~\ref{theorem_fourth_moment} and~\ref{asymptotic_variance} are interlaced. The skeleton of the argument is as follows: 
\begin{enumerate}
\item We first prove convergence of the second moment: $\E \left\langle (R_{1,*} - q^*)^2\right\rangle_t \le \bigo(1/N + e^{-c(t)N})$.
\item We then deduce from 1.\ the convergence of the fourth moment via an inductive argument: $\E \left\langle (R_{1,*} - q^*)^4\right\rangle_t \le \bigo(1/N^2 + e^{-c(t)N})$. This finishes the proof of Theorem~\ref{theorem_fourth_moment}.
\item Using 2., we revisit our proof of 1., and refine the estimates in order to obtain the sharper result $N \cdot \E \left\langle (R_{1,*} - q^*)^2\right\rangle_t \to \Delta_{\RS}(\lambda;t)$. This finishes the proof of Theorem~\ref{asymptotic_variance}.    
\end{enumerate}
We will start by defining our interpolating Hamiltonian and state some preliminary bounds and properties. Then we will move on to the execution of the cavity computations.      

\subsection{Preliminary bounds}
In this section the parameter $t \in [0,1]$ is fixed and treated as a constant. We consider the Hamiltonian 
\begin{align*}
-H_t^{-}(\vct{x}) &:= \sum_{i < j \le N-1} -\frac{\lambda t}{2N} x_i^2x_j^2 + \sqrt{\frac{\lambda t}{N}} W_{ij}x_ix_j + \frac{\lambda t}{N}x_ix_i^*x_jx_j^*\\
&~+\sum_{i=1}^{N-1} -\frac{(1-t)r}{2} x_i^2 +  \sqrt{(1-t)r} z_{i}x_i +  (1-t)r x_ix_i^*,\\
\end{align*}
where we have striped away the contribution of the variable $x_N$ from $H_t$ (equ.\ \eqref{interpolating_hamiltonian}). This contribution is considered separately: for $t' \in [0,1]$, we let
\begin{align*}
-h_{t'}(\vct{x}) &:=  \sum_{i =1}^{N-1}-\frac{\lambda t'}{2N} x_i^2 x_N^2 + \sqrt{\frac{\lambda t'}{N}} W_{iN}x_i x_N + \frac{\lambda t'}{N} x_ix_i^* x_Nx_N^*\\
&~~~-\frac{(1-t')r}{2}x_N^2 +  \sqrt{(1-t')r}z_Nx_N +  (1-t')r x_Nx_N^*.
\end{align*}
We let $r = \lambda q^*(\lambda)$ and let our interpolation, with the time parameter $s \in [0,1]$, be 
\[H_{t,s}(\vct{x}) := H_t^{-}(\vct{x}) + h_{ts}(\vct{x}).\]
At $s=1$ we have $H_{t,s} = H_t$, and at $s=0$ the variable $x_N$ decouples from the rest of the variables. 
For an integer $n\ge 1$ and $f : (\R^{N})^{n+1} \mapsto \R$, we define 
 \begin{equation*}
 \left\langle f(\vct{x}^{(1)},\cdots,\vct{x}^{(n)},\vct{x}^*)\right\rangle_{t,s} := \frac{\int f(\vct{x}^{(1)},\cdots,\vct{x}^{(n)},\vct{x}^*) \prod_{l=1}^n e^{- H_{t,s}(\vct{x}^{(l)})} \rmd P_{\x}^{\otimes N}(\vct{x}^{(l)})}{\int \prod_{l=1}^n e^{- H_{t,s}(\vct{x}^{(l)})}  \rmd P_{\x}^{\otimes N}(\vct{x}^{(l)})},
 \end{equation*}
 similarly to~\eqref{gibbs_average_one}.
Following Talagrand's notation, we write 
\begin{align*}
R^{-}_{l,l'} = \frac{1}{N}\sum_{i=1}^{N-1} x_i^{(l)}x_i^{(l')}, \quad \mbox{and} \quad  \nu_s(f) = \E\langle f\rangle_{t,s}.
\end{align*}
In our last notation, we have only emphasized the dependence of the average on $s$; the parameter $t$ will henceforth remain fixed. 
Moreover, we write $\nu(f)$ for $\nu_1(f)$.  The following three lemmas are variants of Lemma 1.6.3, Lemma 1.6.4 and Proposition 1.8.1 respectively in~\citep{talagrand2011mean1}.  
\begin{lemma}\label{gibbs_derivative}
For all $n \ge 1$,
\begin{align*}
\frac{\rmd}{\rmd s}\nu_s(f) &= \frac{\lambda t}{2} \sum_{1\le l\neq l' \le n} \nu_s((R^{-}_{l,l'}-q) y^{(l)}y^{(l')}f) 
- \lambda t n \sum_{l=1}^n \nu_s(R^{-}_{l,n+1}-q) y^{(l)}y^{(n+1)}f) \\
&~~+ \lambda t n \sum_{l=1}^n \nu_s((R^{-}_{l,*}-q) y^{(l)}y^{*}f) 
- \lambda t n  \nu_s((R^{-}_{n+1,*}-q) y^{(n+1)}y^{*}f) \\
&~~+  \lambda t \frac{n(n+1)}{2} \nu_s((R^{-}_{n+1,n+2}-q) y^{(n+1)}y^{(n+2)}f),
\end{align*}
where we have written $y = x_N$.
\end{lemma}
\begin{proof}
The computation relies on Gaussian integration by parts. See~\cite[Lemma 1.6.3]{talagrand2011mean1} for the details of a similar computation.  
\end{proof}

\begin{lemma}\label{bound_gibbs_derivative}
If $f$ is a bounded non-negative function, then for all $s \in [0,1]$,
\[\nu_s(f) \le K(\lambda ,n) \nu(f).\]
\end{lemma}

\begin{proof}
Since the variables and the overlaps are all bounded, and $t \le 1$, using Lemma~\ref{gibbs_derivative} we have for all $s\in [0,1]$
\[|\nu_s'(f)| \le K(\lambda ,n) \nu_s(f).\] 
Then we conclude using Gr\"onwall's lemma. 
\end{proof}

\begin{lemma}\label{first_second_order_estimates}
For all $s \in [0,1]$, and all $\tau_1,\tau_2 >0$ such that $1/\tau_1+1/\tau_2=1$,
\begin{align}
\abs{\nu_s(f) - \nu_0(f)} &\le K(\lambda,n) \nu\left(\abs{R^-_{1,*} -q}^{\tau_1}\right)^{1/\tau_1} \cdot \nu\left(|f|^{\tau_2}\right)^{1/\tau_2} \label{first_order_estimate}\\
\abs{\nu_s(f) -\nu_0(f) - \nu_0'(f)} &\le K(\lambda,n) \nu\left(\abs{R^-_{1,*} -q}^{\tau_1}\right)^{1/\tau_1} \cdot \nu\left(|f|^{\tau_2}\right)^{1/\tau_2}.\label{second_order_estimate}
\end{align}
\end{lemma}

\begin{proof}
We use Taylor's approximations
\begin{align*}
\abs{\nu_s(f) - \nu_0(f)} &\le \sup_{0 \le s \le 1} \abs{\nu'_s(f)},\\
\abs{\nu_s(f) -\nu_0(f) - \nu_0'(f)} &\le \sup_{0 \le s \le 1} \abs{\nu''_s(f)},
\end{align*}
then Lemma~\ref{gibbs_derivative} and the triangle inequality to bound the right hand sides, then H\"older's inequality to bound each term in the derivative, and then we apply Lemma~\ref{bound_gibbs_derivative}. (To compute the second derivative, one need to use Lemma~\ref{gibbs_derivative} recursively.)  
\end{proof}

\subsection{The cavity matrix} 
Recall the parameters $a(0)$, $a(1)$ and $a(2)$ from~\eqref{4_replica_overlaps}: 
\begin{align*}
a(0) = \E \left[\langle x^2\rangle_r^2\right] - q^{*2}(\lambda), \quad
a(1) = \E \left[\langle x^2 \rangle_r \langle x \rangle_r^2\right]  - q^{*2}(\lambda), \quad
a(2) = \E \left[\langle x\rangle_r^4\right]  - q^{*2}(\lambda),
\end{align*} 
where $r = \lambda q^*(\lambda)$. Now let
\begin{align}\label{cavity_matrix_2}
\mtx{A} := \lambda \cdot
\begin{bmatrix}
a(0) & -2a(1) & a(2)\\
a(1) & a(0) -a(1) -2a(2) & -2a(1) + 3a(2)\\
a(2) & 4a(1) - 6a(2) & a(0)-6a(1)+6a(2)
\end{bmatrix}\cdot
\end{align}
One can easily check that the transpose of this matrix has two eigenvalues $\mu_1$ and $\mu_2$ with expressions
\begin{gather}\label{eigenvalues_2}
\begin{aligned}
\mu_1(\lambda) &= \lambda (a(0) -2a(1)+a(2)),\\
\mu_2(\lambda) &= \lambda (a(0) -3a(1)+2a(2)),
\end{aligned}
\end{gather}
and associated eigenvectors $(1,-2,1)$ and $(2,-3,2)$, and of multiplicities two and one respectively. (The first eigenvalue appears in a $2\times 2$ Jordan block.) We will need to control the largest eigenvalue of $\mtx{A}^\top$. This matrix is the ``planted'' analogue of the one displayed in~\citep[equ.\ (1.234)]{talagrand2011mean1} for the SK model. 
 By Cauchy-Schwarz, $\mu_1-\mu_2 = \lambda (a(1)-a(2)) \ge 0$. As will be clear from the next subsection, the cavity computations we are about present are only informative when $\mu_1 <1$. Interestingly, this is true for all values of $\lambda$ where the $\RS$ formula $\phi_{\RS}$ has two derivatives:
\begin{lemma}\label{latala_guerra}
For all $\lambda \in \A$, $\mu_1(\lambda) < 1$.
\end{lemma} 
\begin{proof}
First, if $\lambda < \lambda_c$, then $q^*(\lambda) = 0$, and $\mu_1(\lambda)= \lambda (\E_{P_{\x}}[X^2])^2$. By Lemma~\ref{spectral_bound}, $\mu_1(\lambda)<1$. Now we assume $\lambda \in \A  \cap (\lambda_c,+\infty)$. Recall 
\[\psi(r) = \E_{x^*,z} \log \int \exp\left(\sqrt{r}zx + r xx^* - \frac{r}{2} x^2\right) \rmd P_{\x}(x),\]
and the $\RS$ potential
\[F(\lambda,q) = \psi(\lambda q) - \frac{\lambda q^2}{4}.\]
 It is a straightforward exercise to compute the first and the second derivatives of $\psi$ using Gaussian integration by parts and the Nishimori property:
\[\psi'(r) = \half \E \langle xx^*\rangle_r, \]
\[\psi''(r) = \half \left(\E\langle x^2x^{*2}\rangle_r -2\E\langle {x^{(1)}}^2x^{(2)}x^{*}\rangle_r + \E\langle x^{(1)}x^{(2)}x^{(3)}x^*\rangle_r \right).\]
With the choice $r = \lambda q^*(\lambda)$, we see that $\mu_1(\lambda) = 2\lambda\psi''(r)$.   
Now we observe that 
\[\frac{\partial^2 F}{\partial q^2}(\lambda,q) = \frac{\lambda}{2}(2\lambda\psi''(\lambda q)-1).\]%
Since $q^*(\lambda)$ is a maximizer of the smooth function $F(\lambda,\cdot)$, and lies in the interior of its domain ($q^*(\lambda) >0$ for $\lambda > \lambda_c$), then it must be a first-order stationary point: $\frac{\partial F}{\partial q}(\lambda,q^*) = 0$. Hence $\frac{\partial^2 F}{\partial q^2}(\lambda,q^*) \le 0$, i.e., $\mu_1(\lambda) \le 1$ for all $\lambda > \lambda_c$. Now we claim that the inequality must be strict for $\lambda \in \A$. Indeed, \cite[Proposition 15]{lelarge2016fundamental} show that whenever $\phi_{\RS}$ is differentiable at $\lambda$, then the maximizer of $F(\lambda,\cdot)$ is unique and
\[\phi_{\RS}'(\lambda) = \frac{q^{*2}(\lambda)}{4}.\]  
Therefore, \emph{twice} differentiability of $\phi_{\RS}$ implies first differentiability of $\lambda \mapsto q^*(\lambda)$ whenever $q^*(\lambda) >0$ (i.e., $\lambda > \lambda_c$). Now we take advantage of first-order optimality: $\frac{\partial F}{\partial q}(\lambda,q^*)  = 0$ is the same as
\[\psi'(\lambda q^*(\lambda)) = \frac{q^*(\lambda)}{2}.\] 
The above can be seen as an equality of functions (of $\lambda$) defined almost everywhere. Taking one derivative yields
\[q^*(\lambda)\psi''(\lambda q^*(\lambda)) = \half \big(1-2\lambda \psi''(\lambda q^*(\lambda)) \big) q^{*'}(\lambda). \] 
Since $q^*(\lambda)$ and $\psi''(\lambda q^*(\lambda))$ are both positive, the right-hand side cannot vanish. This concludes the proof. 
\end{proof}

\subsection{Cavity computations for the second moment}
In this subsection we prove the convergence of the second moment of the overlaps: 
\[\nu((R_{1,*} - q^*)^2) \le \frac{K}{N} + Ke^{-c(t)N},\]
with $c(t) \sim c_0 (1-t)^2$ as $t \to 1$ when $\lambda > \lambda_c$ and $P_{\x}$ is symmetric about the origin, and uniformly lower-bounded by a positive constant otherwise.
To lighten the notation in the calculations to come, $q^*(\lambda)$ will be denoted simply by $q$, and we recall the notation $\nu (\cdot) = \E\langle \cdot \rangle_{t,1}$. 
Let 
\begin{align*}
A = \nu \left((R_{1,*}-q)^2\right),~~~
B = \nu \left((R_{1,*}-q)(R_{2,*}-q)\right),~~~
C = \nu \left((R_{1,*}-q)(R_{2,3}-q)\right).
\end{align*}
By symmetry between sites, 
\[A = \nu \left((R_{1,*}-q)(x_Nx_N^*-q)\right) = \frac{1}{N}\nu\left( x_Nx_N^*(x_Nx_N^*-q)\right) + \nu ( (R^-_{1,*}-q)(x_Nx_N^*-q)).\]
By the first bound~\eqref{first_order_estimate} of Lemma~\ref{first_second_order_estimates} with $\tau_1 = 1$, $\tau_2 = \infty$, we get
\[\nu( x_Nx_N^*(x_Nx_N^*-q)) =  \nu_0(x_Nx_N^*(x_Nx_N^*-q)) + \delta = a(0) + \delta,\]
with $|\delta| \le K(\lambda)\nu(|R^-_{1,*}-q|)$.
On the other hand, by the second bound~\eqref{second_order_estimate} with $\tau_1 = 1$, $\tau_2 = \infty$, we get
\[\nu ((R^-_{1,*}-q)(x_Nx_N^*-q)) = \nu_0'((R^-_{1,*}-q)(x_Nx_N^*-q)) + \delta. \]
This is because $\nu_0((R^-_{1,*}-q)(x_Nx_N^*-q))=0$, since last variable $x_N$ decouples from the remaining $N-1$ variables under the measure $\nu_0$. Now, we use Lemma~\ref{gibbs_derivative} with $n=1$, to evaluate the above derivative at $t=0$. We still write $y^{(l)} = x_N^{(l)}$.
\begin{align*}
\nu_0'((R^-_{1,*}-q)(x_Nx_N^*-q)) &= -\lambda t \nu_0(y^{(1)}y^{(2)}(y^{(1)}y^* - q)(R^-_{1,*}-q)(R^-_{1,2}-q)) \\
&~~+ \lambda t \nu_0(y^{(1)}y^{*}(y^{(1)}y^{*} - q)(R^-_{1,*}-q)^2)\\
&~~ - \lambda t \nu_0(y^{(2)}y^{*}(y^{(1)}y^* - q)(R^-_{1,*}-q)(R^-_{2,*}-q)) \\
&~~+ \lambda t \nu_0(y^{(2)}y^{(3)}(y^{(1)}y^* - q)(R^-_{1,*}-q)(R^-_{2,3}-q)).
\end{align*}
We extract the average on the $y$-variables from the rest of the expression as pre-factors, so that the above is equal to
\begin{align*}
&-\lambda t a(1) \nu_0((R^-_{1,*}-q)(R^-_{1,2}-q))
+ \lambda t a(0) \nu_0((R^-_{1,*}-q)^2)\\
&- \lambda t a(1)\nu_0((R^-_{1,*}-q)(R^-_{2,*}-q))
+ \lambda t a(2)\nu_0((R^-_{1,*}-q)(R^-_{2,3}-q)).
\end{align*}
We notice that by the Nishimori property that \[\nu_0((R^-_{1,*}-q)(R^-_{1,2}-q)) = \nu_0((R^-_{1,*}-q)(R^-_{2,*}-q)).\]
Now we observe that $\nu_0'((R^-_{1,*}-q)(x_Nx_N^*-q))$ is a linear combination of terms that resemble, but are not quite equal to $A$, $B$ and $C$. We are nevertheless tempted to make the substitution since we expect them to be close. We use Lemma~\ref{first_second_order_estimates} to justify this. Taking $\nu_0((R^-_{1,*}-q)^2)$ as an example, we apply the estimate~\eqref{first_order_estimate} with $t=1$, $\tau_1 = 3$ and $\tau_2 = 3/2$. We get
\[\nu_0((R^-_{1,*}-q)^2) = \nu((R^-_{1,*}-q)^2) + \delta\]
with $|\delta| \le K(\lambda)\nu(|R^-_{1,*}-q|^3)$. Moreover,
\[\nu((R^-_{1,*}-q)^2) = \nu((R_{1,*}-\frac{1}{N}yy^*-q)^2) = \nu((R_{1,*}-q)^2) -\frac{2}{N}\nu(yy^*(R_{1,*}-q)) + \frac{1}{N^2}\nu({y}^2{y^*}^2).\]
The third term is of order $1/N^2$, and the second term is bounded by $\frac{1}{N}\nu_0(\abs{R_{1,*}-q})$. Therefore
\[\nu_0((R^-_{1,*}-q)^2) = \nu((R_{1,*}-q)^2) + \delta',\] 
with \[|\delta'| \le K(\lambda)\left(\frac{1}{N}\nu(|R^-_{1,*}-q|)+\nu(|R^-_{1,*}-q|^3) + \frac{1}{N^2}\right).\]
This argument applies equally to the remaining terms $\nu_0((R^-_{1,*}-q)(R^-_{2,*}-q))$ and $\nu_0((R^-_{1,*}-q)(R^-_{2,3}-q))$. We then end up with the identity
\begin{equation}\label{A_eqn}
A = \frac{a(0)}{N} + \lambda' a(0)A  -2\lambda' a(1)B + \lambda' a(2) C + \delta(0),
\end{equation}
where $\lambda' = t \lambda$, and $|\delta(0)|$ is bounded by the same quantity as $|\delta'|$.

Next, we apply the same reasoning to $B$ and $C$ as well, (e.g., Lemma~\ref{gibbs_derivative} needs to applied with $n=2$ for $B$ and $n=3$ for $C$) we get
\begin{align}
B &= \frac{a(1)}{N} + \lambda' a(1) A + \lambda' (a(0)-a(1)-2a(2))B +\lambda' (-2a(1)+3a(2))C + \delta(1),\label{B_eqn}\\  
C &= \frac{a(2)}{N} + \lambda' a(2)A + \lambda' (4a(1)-6a(2))B + \lambda' (a(0)-6a(1)+6a(2))C + \delta(2),\label{C_eqn}
\end{align}
where for $i=0,1,2$, 
\begin{equation}\label{delta_error}
|\delta(i)| \le  K(\lambda)\left(\frac{1}{N}\nu(|R^-_{1,*}-q|)+\nu(|R^-_{1,*}-q|^3) + \frac{1}{N^2}\right).
\end{equation}
We have ended up with a linear system in the quantities $A$, $B$ and $C$. Let $\vct{z} = [A,B,C]^\top$ and $\vct{\delta} = [\delta(0),\delta(1),\delta(2)]^\top$. Then the equations~\eqref{A_eqn}, \eqref{B_eqn} and \eqref{C_eqn} can be written as 
\begin{equation}\label{cavity_linear_system}
\vct{z} = \frac{1}{N}\vct{a}+t\mtx{A}\vct{z} + \vct{\delta},
\end{equation}
where $\vct{a} = [a(0),a(1),a(2)]^\top$, and the matrix $\mtx{A}$ are defined in~\eqref{cavity_matrix_2}.
The above system implies useful bounds on the coefficients of the vector $\vct{z}$ only if the largest eigenvalue of the matrix $t\mtx{A}$ is smaller than 1. This is insured by Lemma~\ref{latala_guerra} when $\lambda \in \A$ (independently of $t$). Now we can invert the linear system and extract $\vct{z}$:
\begin{equation}\label{cavity_linear_system_2}
\vct{z} = \frac{1}{N}(\mtx{I} - t\mtx{A})^{-1}\vct{a} +  (\mtx{I} - t\mtx{A})^{-1}\vct{\delta}.
\end{equation}
Now we need to control the entries of $\vct{\delta}$. By elementary manipulations,
\[\nu(|R_{1,*}^- - q|) \le \nu(|R_{1,*} - q|) + \frac{K}{N},\] 
and
\[\nu(|R_{1,*}^- - q|^3) \le \nu(|R_{1,*} - q|^3) + \frac{K}{N}\nu((R_{1,*} - q)^2) + \frac{K}{N^2}\nu(|R_{1,*} - q|) + \frac{K}{N^3}.\] 
Therefore, from~\eqref{delta_error} we have for all $i=0,1,2$,
\begin{equation}\label{delta_error_2}
|\delta(i)| \le K \left( \nu(|R_{1,*} - q|^3) +  \frac{1}{N}\nu((R_{1,*} - q)^2) + \frac{1}{N}\nu(|R_{1,*} - q|) +\frac{1}{N^2} \right).
\end{equation}
Now we will argue that $\nu(|R_{1,*} - q|) \ll 1$ and $\nu(|R_{1,*} - q|^3) \ll \nu((R_{1,*} - q)^2)$. With Lemma~\ref{convergence_in_probability} we have for $\epsilon>0$
\[\nu(|R_{1,*} - q|)  \le \epsilon + K(\epsilon) e^{-cN},\]
and
\[\nu( |R_{1,*} - q|^3)  \le \epsilon \nu( (R_{1,*} - q)^2) + K(\epsilon) e^{-cN}.\]
Combining the above two bounds with~\eqref{delta_error_2}, and then injecting in~\eqref{cavity_linear_system_2}, we get
\begin{align*}
 \nu((R_{1,*} - q)^2)  = z(0) &\le \twonorm{z} \le \twonorm{\frac{1}{N}(\mtx{I} - t\mtx{A})^{-1}\vct{a}} + \left\|(\mtx{I} - t\mtx{A})^{-1}\right\|_{\text{op}} \twonorm{\vct{\delta}}\\
  &\le \frac{\twonorm{\vct{c}}}{N} + K(\epsilon+\frac{1}{N}) \nu( (R_{1,*} - q)^2) + K(\epsilon)e^{-cN}.
\end{align*}
The symbols $\twonorm{\cdot}$ and $\|\cdot \|_{\text{op}}$ refer to the $\ell_2$ norm of a vector and the matrix operator norm respectively. Here, $\vct{c} = (\mtx{I} - t\mtx{A})^{-1}\vct{a}$. Note that the matrix inverses are bounded even as $t \to 1$ since $\mu_1 <1$ for $\lambda \in \A$. We choose $\epsilon$ small enough and $N$ large enough that $K(\epsilon+\frac{1}{N}) <1$. We therefore get
\[\nu\left((R_{1,*} - q)^2\right) \le \frac{K(\lambda)}{N} + K(\lambda)e^{-c(t)N}.\] 

\subsection{Cavity computations for the fourth moment}
In this subsection we prove the convergence of the fourth moment: 
\[\nu((R_{1,*} - q^*)^4) \le \frac{K}{N^2} + Ke^{-c(t)N},\]
where $c(t)$ is of the same type as before.
We adopt the same technique based on the cavity method, with the extra knowledge that the second moment converges.
Many parts of the argument are exactly the same so we will only highlight the main novelties in the proof. Let
\begin{align*}
A = \nu \left((R_{1,*}-q)^4\right),~~~
B = \nu \left((R_{1,*}-q)^3(R_{2,*}-q)\right),~~~
C = \nu \left((R_{1,*}-q)^3(R_{2,3}-q)\right).
\end{align*}
By symmetry between sites, 
\begin{align*}
A &= \nu \left((R_{1,*}-q)^3(x_Nx_N^*-q)\right)\\ 
&=  \nu((R^-_{1,*}-q)^3(x_Nx_N^*-q)) + \frac{3}{N}\nu ((R^-_{1,*}-q)^2 x_Nx_N^* (x_Nx_N^*-q))\\
 &~~~+ \frac{3}{N^2}\nu ((R^-_{1,*}-q){x_N}^2{x_N^*}^2(x_Nx_N^*-q)) 
 + \frac{1}{N^3}\nu({x_N}^3{x_N^*}^3(x_Nx_N^*-q)).
\end{align*}
The quadratic term is bounded as
\[\nu( (R^-_{1,*}-q)^2 x_Nx_N^* (x_Nx_N^*-q)) \le K \nu((R^-_{1,*}-q)^2) \le \frac{K}{N} + Ke^{-cN}. \]  
The last inequality is using our extra knowledge about the convergence of the second moment. The last two terms are also bounded by $K/N^2$ and $K/N^3$ respectively. 
Now we must deal with the cubic term, and here, we apply the exact same technique used to deal with the term $\nu ((R^-_{1,*}-q)(x_Nx_N^*-q))$ in the previous proof. The argument goes verbatim. Then we equally treat the terms $B$ and $C$. We end up with a similar linear system relating $A$, $B$ and $C$:
\begin{equation*}
\vct{z} = \frac{1}{N^2}\vct{d}+t\mtx{A}\vct{z} + \vct{\delta},
\end{equation*}
where $\vct{z} = [A,B,C]^\top$. The differences with  the earlier linear system~\eqref{cavity_linear_system} are in the vector of coefficients $\vct{d}$ (that could be determined from the recursions) and the error terms $\delta(i)$, which are now bounded as
\[|\delta(i)| \le K\nu(|R^-_{1,*}-q|^{5})+K\sum_{l=1}^3 \frac{1}{N^{3-l}}\nu(|R^-_{1,*}-q|^{l}).\]
Crucially, the matrix $\mtx{A}$ remains the same.
Using Lemma~\ref{convergence_in_probability}, we have for $\epsilon >0$,
\[\nu( |R_{1,*} - q|^5)  \le \epsilon \nu( (R_{1,*} - q)^4) + K(\epsilon) e^{-cN},\]
\[\nu( |R_{1,*} - q|^3)  \le \epsilon \nu( (R_{1,*} - q)^2) + K(\epsilon) e^{-cN}.\]
With the bound we already have on $\nu( (R_{1,*} - q)^2)$, we finish the argument in the same way, by choosing $\epsilon$ sufficiently small. 
This concludes the proof of Theorem~\ref{theorem_fourth_moment}.

\subsection{Sharper results: the asymptotic variance}
Finally, given the convergence of the fourth moment, we can refine the convergence result of the second moment. Indeed we are now able to compute the limit of $N\cdot\nu((R_{1,*} - q)^2)$. 
Using Jensen's inequality on the second and fourth moment, we have
\[\nu(|R_{1,*} - q|) \le \frac{K}{\sqrt{N}} + Ke^{-c(t)N}, ~~~ \mbox{and} ~~~\nu(|R_{1,*} - q|^3) \le \frac{K}{N^{3/2}} + Ke^{-c'(t)N}.\]
Looking back at~\eqref{delta_error_2}, we can now assert that
\[|\delta(i) | \le \frac{K}{N^{3/2}} + Ke^{-c(t)N}.\]
We plug this new knowledge in~\eqref{cavity_linear_system_2}, and obtain
\begin{align*}
\twonorm{N\vct{z}-(\mtx{I} - t\mtx{A})^{-1}\vct{a}} &\le N\left\|(\mtx{I} - t\mtx{A})^{-1}\right\|_{\text{op}}\twonorm{\vct{\delta}} \le K\Big(\frac{1}{\sqrt{N}} + N e^{-c(t)N}\Big).
\end{align*}
The last line follows since $\sup_{t}\left\|(\mtx{I} - t\mtx{A})^{-1}\right\|_{\text{op}} \le K(\lambda)$ for $\lambda \in \A$.
We have just proved that
\begin{align*}
\nu((R_{1,*} - q)^2) &= \frac{c(0)}{N} + K(\lambda)\Big(\frac{1}{N^{3/2}} + e^{-c(t)N}\Big),\\
\nu((R_{1,*} - q)(R_{2,*} - q)) &= \frac{c(1)}{N} + K(\lambda)\Big(\frac{1}{N^{3/2}} + e^{-c(t)N}\Big),\\
\nu((R_{1,*} - q)(R_{2,3} - q)) &= \frac{c(2)}{N} + K(\lambda)\Big(\frac{1}{N^{3/2}} + e^{-c(t)N}\Big),
\end{align*}
where $\vct{c} = (\mtx{I} - t\mtx{A})^{-1}\vct{a}$. One can solve this linear system explicitly and obtain the expression of the coordinates of $\vct{c}$:
\begin{align*}
c(0) &= \frac{1}{\lambda t} \left(-1+ \frac{2}{1-t\mu_2} + \frac{2}{1-t\mu_1} +\frac{-3+3\lambda ta(0) -2\lambda ta(1)}{(1-t\mu_1)^2} \right),\\
c(1) &=   \frac{1}{\lambda t} \left(\frac{-3+3\lambda ta(0) -2\lambda ta(1)}{(1-t\mu_1)^2} + \frac{3}{1-t\mu_2}\right),\\
c(2) &=  \frac{4\lambda t a(1)^2 + (1- \lambda t a(0) - 5 \lambda t a(1))a(2) + 2 \lambda ta(2)^2}{(1-t\mu_1)^2(1-t\mu_2)}.
\end{align*}
The expression of the first coordinate defines $\Delta_{\RS}(\lambda; t)$, equ.\eqref{Delta_RS}. This concludes the proof of Theorem~\ref{asymptotic_variance}.

\subsection{Proof of Lemma~\ref{diagonal_term}}
Let $f(x,x^*) =  x^2 x^{*2}$. We have $\nu_0(f) = a(0)$. We use the first assertion of Lemma~\ref{first_second_order_estimates} with $\tau_1 = 1$ and $\tau_2 = \infty$ to get
\[\abs{\nu(f) -\nu_0(f)} \le K(\lambda) \nu(|R_{1,*}^- - q^*|) \le \frac{K}{\sqrt{N}} + Ke^{-c(t)N},\]
where the last bound follows from Theorem~\ref{theorem_fourth_moment} and Jensen's inequality.

%%%%%%%%%%%%%%%%%%%%%%%%%%%%%%
\section{Proof of Theorem~\ref{central_limit_theorem}}
\label{sxn:proof_fluctuations}
In this section we prove a slightly stronger result than convergence in distribution. We prove the convergence of all moments with an explicit rate of $\bigo(N^{-1/2})$. Statement $(ii)$ is deduced effortlessly form a classical result while statement $(i)$ requires more work. We start with the former. 

\subsection{Fluctuations under $\P_{0}$: the ALR CLT}
We assume in this subsection that $P_{\x} = \half \delta_{-1} + \half \delta_{+1}$, and let $\mtx{Y} \sim \P_{0}$, i.e, $Y_{ij} \sim \normal(0,1)$ i.i.d. We then see that the likelihood ratio is related to the partition function of the Sherrington--Kirkpatrick model via a trivial relation:  
\begin{align*}
\log L(\mtx{Y};\lambda) &= \log \int \exp\Big(\sqrt{\frac{\lambda}{N}}\sum_{i < j} Y_{ij}x_ix_j -\frac{\lambda}{2N}\sum_{i < j} x_i^2x_j^2\Big) ~\rmd P_{\x}^{\otimes N}(\vct{x})\\ 
&= \log \sum_{\vct{\sigma} \in \{\pm 1\}^N} \exp\Big(\frac{\beta}{\sqrt{N}}\sum_{i < j} Y_{ij}\sigma_i\sigma_j \Big) - N \log 2 - \frac{\beta^2 (N-1)}{4}\\
&=: \log Z_N(\beta) - N \log 2 - \frac{\beta^2 (N-1)}{4},
\end{align*}
where we have let $\beta = \sqrt{\lambda}$. $Z_N(\beta)$ is the partition function of the SK model at inverse temperature $\beta>0$. It is easy to compute the expectation of $Z_N(\beta)$:
\[\log \E Z_N(\beta) = N\log 2 + \frac{\beta^2 (N-1)}{4},\]
so that $\E\log L(\mtx{Y};\lambda)$ is the gap between the free energy of the SK model and its \emph{annealed} version. The question of determining the values of the inverse temperature for which this gap is zero (or constant), i.e., \emph{at what temperatures is the free energy given by the annealed computation?} is a central question in statistical physics. 
Aizenman, Lebowitz and Ruelle (ALR) proved that in the high-temperature regime $\beta <1$, $\log (Z_N(\beta)/\E Z_N(\beta))$ converges in distribution the normal law 
\[\normal\left(\frac{1}{4}(\log(1-\beta^2) + \beta^2) , -\frac{1}{2}(\log(1-\beta^2) + \beta^2) \right).\]
In our notation this simply means
\[\log L(\mtx{Y};\lambda) \rightsquigarrow \normal(-\mu, \sigma^2)\]     
under $\P_{0}$ where $\mu = \half \sigma^2 =  -\frac{1}{4}(\log(1-\lambda) + \lambda)$. The ALR proof is combinatorial and uses so-called \emph{cluster expansion} techniques. It may not extend to other types of priors. Alternative proofs were subsequently found by adopting different perspectives on the problem~\citep{comets1995sherrington,guerra2002central}. A more recent proof based on the cavity method is provided by Talagrand in his second book~\citep[Section 11.4]{talagrand2011mean2}. His method provides an explicit (and optimal) rate of convergence of the moments of the random variable in question to those of the Gaussian. In what follows we use Talagrand's approach to prove a similar central limit theorem when $\mtx{Y} \sim \P_{\lambda}$, for an arbitrary bounded prior $P_{\x}$. In this more general setting, the high temperature region of the model is given by the condition $\lambda<\lambda_c$.   

\subsection{Fluctuations under $\P_{\lambda}$: a planted version of the ALR CLT}
Let $\lambda<\lambda_c$, and $\mtx{Y} \sim \P_{\lambda}$. We define the random variable
\[X(\lambda) = \log L(\mtx{Y}) - \mu(\lambda),\]
where
\[\mu(\lambda) = \frac{1}{4}(-\log (1-\lambda) - \lambda), \quad \mbox{and} \quad b(\lambda) = \sigma^2(\lambda) = 2\mu(\lambda).\]
We will prove that the integer moments of $X(\lambda)$ converge to those of the Gaussian with variance $b(\lambda)$. This is a sufficient condition for convergence in distribution to hold, since the Gaussian is uniquely determined by its moments.

\begin{theorem}\label{moment_inequality}
For all $\lambda<\lambda_c$ and integers $k$, there exists a constant $K(\lambda,k) \ge 0$ such that 
\[\abs{\E\left[X(\lambda)^k \right] - m(k)b(\lambda)^{k/2}}\le \frac{K(\lambda,k)}{\sqrt{N}},\]
where $m(k) = \E[g^k]$ is the $k$-th moment of the standard Gaussian $g \sim \normal(0,1)$. 
\end{theorem}
This theorem mirrors Theorem 11.4.1 in~\citep{talagrand2011mean2}, and our approach is inspired by his.  
We define the function
\[f(\lambda):= \E\left[X(\lambda)^k \right].\] 
\begin{lemma}
For all $\lambda<\lambda_c$,
\begin{align} \label{derivative_f}
f'(\lambda) &= -\frac{k}{4}\E\left[\left(N\langle R_{1,2}^2\rangle - \langle x_N^2\rangle^2\right)X(\lambda)^{k-1}\right] + \frac{k}{2}\E\left[\left(N\langle R_{1,*}^2\rangle - \langle x_N^2x_N^{*2}\rangle\right)X(\lambda)^{k-1}\right]\nonumber\\
&~~~-k\mu'(\lambda)\E\left[X(\lambda)^{k-1}\right]
+\frac{k(k-1)}{4}\E\left[\left(N\langle R_{1,2}^2\rangle 
- \langle x_N^2\rangle^2\right)X(\lambda)^{k-2}\right].
\end{align}
\end{lemma}
\begin{proof}
This is by simple differentiation and regrouping of terms.
\end{proof}

The derivative involves averages of the form 
\[\E\left[\left(N\langle R_{1,l}^2\rangle - \langle {x^{(1)}_N}^2{x^{(l)}_N}^2\rangle\right)X(\lambda)^k\right],\]
for $l=2,*$. In the first line of~\eqref{derivative_f}, we see that the planted term $l=*$ has a pre-factor twice as big the that of the replica term $l=2$. This is the reason the mean of the limiting Gaussian is $\mu$ and not $-\mu$ in the planted case.  
A crucial step in the argument is to show that $X(\lambda)^k$ and its pre-factor in the above expression are asymptotically uncorrelated, so that one can split the expectation:
  
\begin{proposition}\label{delicate_overlap_convergence}
For all $\lambda<\lambda_c$ and integers $k \ge 1$, and $l \in \{2,*\}$, we have
\[\E\left[\left(N\langle R_{1,l}^2\rangle - \langle {x^{(1)}_N}^2{x^{(l)}_N}^2\rangle\right)X(\lambda)^k\right] = \frac{\lambda}{1-\lambda}\E\left[ X(\lambda)^k\right] + \delta,\]
 where $|\delta| \le K(k,\lambda)/\sqrt{N}$. 
\end{proposition}

\vspace{.5cm}
\begin{proofof}{Theorem~\ref{moment_inequality}}
Proposition~\ref{delicate_overlap_convergence} implies
\[f'(\lambda) = k\left(\frac{1}{4}\frac{\lambda}{1-\lambda} - \mu'(\lambda)\right)\E\left[X(\lambda)^{k-1}\right] + \frac{k(k-1)}{4}\frac{\lambda}{1-\lambda}\E\left[X(\lambda)^{k-2}\right] + \delta.\]  
Notice that with our choice of the function $\mu$, the first term on the right-hand side vanishes. (Setting this term to zero provides another way of discovering the function $\mu$.) Now we let $b(\lambda) = 2\mu(\lambda)$. We have for all $\lambda$ and all $k \ge 2$
\begin{equation}\label{key_recursion}
\frac{\rmd }{\rmd \lambda} \E\left[X(\lambda)^{k}\right] = \frac{k(k-1)}{2}b'(\lambda)\E\left[X(\lambda)^{k-2}\right] + \delta.
\end{equation}
By induction, and since $X(0) = 0$, we see that for all even $k$
\[\E\left[X(\lambda)^{k}\right] = m(k)b(\lambda)^{k/2} + \bigo\left(\frac{K(k,\lambda)}{\sqrt{N}}\right),\]
where $m(k) = (k-1)m(k-2)$ and $m(0)=1$. The last recursion defines the sequence of even Gaussian moments. As for odd values of $k$, we have already proved in Corollary~\ref{cor_kul_divergence} that
\[\abs{\E\left[X(\lambda)\right]} \le \frac{K(\lambda)}{\sqrt{N}}.\] 
We use induction again on~\eqref{key_recursion} to conclude that for all odd $k$,
\[\abs{\E\left[X(\lambda)^{k}\right]} \le \frac{K(k,\lambda)}{\sqrt{N}}.\] 
\end{proofof}

\subsection{Proof of Proposition~\ref{delicate_overlap_convergence}}
The argument is in two stages. We first prove that 
\begin{equation}\label{first_fundamental_bound}
N\cdot \E\left[\langle R_{1,l}^2\rangle X(\lambda)^k\right] = \frac{1}{1-\lambda}\E\left[\langle {x^{(1)}_N}^2{x^{(l)}_N}^2\rangle X(\lambda)^k\right] + \delta,
\end{equation}%\bigo\left(\frac{1}{\sqrt{N}}\right)
and then
\begin{equation}\label{second_fundamental_bound}
\E\left[\langle{x^{(1)}_N}^2{x^{(l)}_N}^2\rangle X(\lambda)^k\right] = \E\left[X(\lambda)^k\right] + \delta,
\end{equation}
where in both cases $|\delta| \le K(k,\lambda)/\sqrt{N}$.
We once again use the cavity method to extract the last variable $x_N$ and analyze its influence. Let
\begin{align*}
-H_t(\vct{x}) &:= \sum_{1\le i < j \le N-1} -\frac{\lambda}{2N} x_i^2x_j^2 + \sqrt{\frac{\lambda}{N}} W_{ij}x_ix_j + \frac{\lambda}{N}x_ix_i^*x_jx_j^*\\
&~~~+\sum_{i =1}^{N-1}-\frac{t\lambda}{2N} x_i^2 x_N^2 + \sqrt{\frac{t\lambda}{N}} W_{iN}x_i x_N + \frac{t\lambda}{N} x_ix_i^* x_Nx_N^*,
\end{align*}
and
\[Y(t):= \log \int e^{-H_t(\vct{x})} \rmd P_{\x}^{\otimes N}(\vct{x}) - \mu(\lambda).\]
We have $Y(1) = X(\lambda)$. We consider the quantity 
\[\varphi(t;l) := \E\left[\left(N\langle R_{1,l}^2\rangle_t - \langle {x^{(1)}_N}^2{x^{(l)}_N}^2\rangle_t\right)Y(t)^k\right].\]
Our strategy is approximate $\varphi(t;l)$ by $\varphi(0;l)+\varphi'(0;l)$.  
The approach is very similar to the one used to prove optimal rates of convergence of the overlaps. 
By symmetry between sites, we have
\[\varphi(t;l) =  N\E\left[\left\langle x_N^{(1)}x_N^{(l)}R_{1,l}^-\right\rangle_tY(t)^k\right].\]
Notice that since the last variables decouple from the rest of the system at $t=0$, we have
\begin{align*}
\varphi(0;l) &= N\E\left[\langle x_N^{(1)}x_N^{(l)}\rangle_0\right] \cdot \E\left[\left\langle R_{1,l}^-\right\rangle_0Y(0)^k\right]\\
&= N\E_{P_{\x}}[X]^2 \cdot \E\left[\left\langle R_{1,l}^-\right\rangle_0Y(0)^k\right] = 0.
\end{align*}
The expressions of the derivatives are a bit cumbersome so we do not display them, but we will describe their main features. From here onwards, we present the proof of~\eqref{first_fundamental_bound} and ~\eqref{second_fundamental_bound} for $l=2$ for concreteness. The exact same argument goes through for $l=*$. The only difference is in the number of terms showing up the derivatives of $\varphi$, not their nature.      
The derivative $\varphi'(t;2)$ will be a sum of different terms, all of the form 
\begin{equation}\label{first_derivative_order_two}
\lambda N c(k) \E\left[\left\langle R^-_{1,2}R^-_{a,b} x_N^{(1)}x_N^{(2)}x_N^{(a)}x_N^{(b)}\right\rangle_t Y(t)^n\right],
\end{equation}
where $n\in \{k-2,k-1,k\}$ and $(a,b) \in \{(1,2),(1,3),(3,4),(1,*),(3,*)\}$, and $c(k)$ is a polynomial of degree $\le 2$ in $k$. 
We see that at $t=0$, if the above expression involves a variable $x_N$ of degree 1 then this term vanishes. Therefore the only remaining term is the one where $(a,b) = (1,2)$. One can verify that $c(k) = 1$ for this term. Therefore
\begin{align}
\varphi'(0;2) &= \lambda N  \E\left[\langle {x_N^{(1)}}^2{x_N^{(2)}}^2\rangle_0\right] \cdot \E \left[\langle (R_{1,l}^-)^2 \rangle_0 Y(0)^k \right] \nonumber\\
&= \lambda N  \E_{P_{\x}}\left[X^2\right]^2 \cdot \E \left[\langle (R_{1,l}^-)^2 \rangle_0 Y(0)^k \right]\nonumber\\
&= \lambda N  \E \left[\langle (R_{1,l}^-)^2 \rangle_0 Y(0)^k \right].\label{second_derivative_at_0}
\end{align}

Now we turn to $\varphi''(t;2)$. Taking another derivative generates monomials of degree three in the overlaps and the last variable, so $\varphi''(t;2)$ is a sum of terms of the form
\begin{equation}\label{second_derivative_order_three}
\lambda^2 N c'(k) \E\left[\left\langle R^-_{1,2}R^-_{a,b}R^-_{c,d} x_N^{(1)}x_N^{(2)}x_N^{(a)}x_N^{(b)}x_N^{(c)}x_N^{(d)}\right\rangle_t Y(t)^n\right],
\end{equation}
where $c'(k)$ is a polynomial of degree $\le 3$ in $k$, and $n \in \{k-3,k-2,k-1,k\}$.
Our goal is to bound the second derivative independently of $t$, so that we are able to use the Taylor approximation
\begin{equation}\label{taylor_number_2}
\abs{\varphi(1;2) -\varphi(0;2)-\varphi'(0;2)} \le \sup_{0\le t \le 1}\abs{\varphi''(t;2)}.
\end{equation}
Since prior $P_{\x}$ has bounded support, H\"older's inequality implies that~\eqref{second_derivative_order_three} is bounded by  
\begin{align*} 
N K(k,\lambda) &\E\left[\left\langle \abs{R^-_{1,2}R^-_{a,b}R^-_{c,d}}\right\rangle_t^p\right]^{1/p} \E\left[|Y(t)|^{nq}\right]^{1/q}\\
 &\le N K(k,\lambda) \E\left[\left\langle |R^-_{1,2}|^{3p}\right\rangle_t\right]^{1/p} \E\left[|Y(t)|^{nq}\right]^{1/q},
\end{align*}
where $1/p+1/q=1$. The last bound follows from Jensen's inequality (since $p\ge 1$) and another application of H\"older's inequality. We let $p=4/3$ and $q=4$. Using a straightforward analogue of Lemma~\ref{bound_gibbs_derivative} for the measure $\E\langle \cdot \rangle_t$, and the convergence of the fourth moment, Theorem~\ref{theorem_fourth_moment}, we have 
\[\E\left\langle (R^-_{1,2})^{4}\right\rangle_t \le K(\lambda)\E\left\langle (R^-_{1,2})^{4}\right\rangle \le \frac{K(\lambda)}{N^2}.\]
We use the following lemma to bound the moments of $Y(t)$: 
\begin{lemma}\label{bounded_moments}
For all $\lambda <\lambda_c$ and integers $k$, there exists a constant $K(k,\lambda) \ge 0$ such that for all $t \in [0,1]$
\[\E\left[Y(t)^{2k}\right] \le K(k,\lambda). \]
\end{lemma}
\begin{proof}
Taking a derivative w.r.t.\ time, we have
\begin{align*}
\frac{\rmd}{\rmd t}\E\left[Y(t)^k\right] = &-\frac{\lambda k}{2}\E\left[\left\langle x_N^{(1)}x_N^{(2)}R_{1,2}^-\right\rangle_tY(t)^{k-1}\right] 
+ \lambda k\E\left[\left\langle x_N^{(1)}x_N^{*}R_{1,*}^-\right\rangle_tY(t)^{k-1}\right] \\
&+ \frac{\lambda k(k-1)}{2}\E\left[\left\langle x_N^{(1)}x_N^{(2)}R_{1,2}^-\right\rangle_tY(t)^{k-2}\right].
\end{align*}
By H\"older's inequality and boundedness of the variables and overlaps, 
\[\abs{\frac{\rmd}{\rmd t}\E\left[Y(t)^k\right] } \le K(k,\lambda) \left(\E\left[|Y(t)|^k\right]^{1-1/k} + \E\left[|Y(t)|^k\right]^{1-2/k}\right).\]
The first term is generated by the terms involving $Y(t)^{k-1}$ in the derivative, and the second term comes from the one involving $Y(t)^{k-2}$. Since $k$ is even, we drop the absolute values on the right-hand side. Next, use the fact $x^a \le 1+x$ for all $x\ge 0$ and $0\le a\le1$, then we use Gr\"onwall's lemma to conclude.
\end{proof}

Therefore by the above estimates we have 
\begin{equation}\label{second_derivative_bound}
\sup_{0\le t \le 1}|\varphi''(t;2)| \le \frac{K(k,\lambda)}{\sqrt{N}}.
\end{equation}
Now, our next goal is to prove
\begin{equation}\label{first_derivative_approx}
\abs{ \varphi'(0;2) - \lambda N \E \left[\langle R_{1,2}^2 \rangle X(\lambda)^k \right]} \le\frac{K(k,\lambda)}{\sqrt{N}}.
\end{equation}
We consider the function (this should come with no surprise at this point)
\[\psi(t) := \lambda N \E \left[\left\langle (R_{1,2}^-)^2 \right\rangle_t Y(t)^k \right].\]
Observe that~\eqref{second_derivative_at_0} tells us $\psi(0)=\varphi'(0;2)$. On the other hand, 
\[\abs{\psi(1) -  \lambda N \E \left[\langle R_{1,2}^2 \rangle X(\lambda)^k \right]} \le 2\lambda \E\left[\left\langle \abs{R_{1,2}^- x_N^{(1)}x_N^{(2)}}\right\rangle |X(\lambda)|^k\right] + \frac{\lambda}{N} \E\left[\left\langle {x_N^{(1)}}^2{x_N^{(2)}}^2\right\rangle |X(\lambda)|^k\right].\]
Using Lemma~\ref{bounded_moments} and H\"older's inequality, the first term is bounded by 
$K(k,\lambda) (\E\langle (R_{1,2}^-)^2\rangle)^{1/2} \le K(k,\lambda)/\sqrt{N}$, 
and the second term is bounded by $K(k,\lambda)/N$. So it suffices to show that
\[\sup_{0\le t \le 1}|\psi'(t)| \le \frac{K(k,\lambda)}{\sqrt{N}}.\]
Similarly to $\varphi$, the derivative of $\psi$ is a sum of terms of the form
\[\lambda^2 N c(k) \E\left[\left\langle (R^-_{1,2})^2R^-_{a,b}x_N^{(a)}x_N^{(b)}\right\rangle_t Y(t)^n\right].\]
It is clear that the same method used to bound $\varphi''$ (the generic term of which is~\eqref{second_derivative_order_three}) also works in this case, so we obtain the desired bound on $\psi'$. 
Finally, using~\eqref{taylor_number_2}, \eqref{second_derivative_bound} and \eqref{first_derivative_approx}, we obtain 
\[N\E\left[\langle R_{1,2}^2\rangle X(\lambda)^k\right] - \E\left[\langle {x^{(1)}_N}^2{x^{(2)}_N}^2\rangle X(\lambda)^k\right] = \lambda N \E \left[\langle R_{1,2}^2 \rangle X(\lambda)^k \right] + \delta,\]
where $|\delta| \le K(k,\lambda)/\sqrt{N}$. This is equivalent to~\eqref{first_fundamental_bound} and closes the first stage of the argument.
Now we need to show that
\[\E\left[\langle{x^{(1)}_N}^2{x^{(2)}_N}^2\rangle X(\lambda)^k\right] = \E\left[X(\lambda)^k\right] + \delta.\]
The argument has become a routine by now: we consider the function
\[\psi(t) = \E\left[\langle{x^{(1)}_N}^2{x^{(2)}_N}^2\rangle_t Y(t)^k\right].\]
We have 
\[\psi(0) = \E\left[\langle{x^{(1)}_N}^2{x^{(2)}_N}^2\rangle_0\right] \cdot \E\left[Y(0)^k\right] =  \E_{P_{\x}}[X^2]^2 \cdot \E\left[Y(0)^k\right] = \E\left[Y(0)^k\right].\]
The derivative of $\psi$ is a sum of term of the form
\[\lambda c(k)\E\left[\left\langle{x^{(1)}_N}^2{x^{(2)}_N}^2 R^-_{a,b}x_N^{(a)}x_N^{(b)}\right\rangle_t Y(t)^n\right].\]
By our earlier argument, $|\psi'(t)| \le K(k,\lambda)/\sqrt{N}$ for all $t$. We similarly argue that $\abs{\frac{\rmd}{\rmd t}\E[Y(t)^k]} \le K(k,\lambda)/\sqrt{N}$ for all $t$, so that
\[\abs{\psi(1) - \E\left[Y(1)^k\right]} \le \frac{K(k,\lambda)}{\sqrt{N}}.\]  
This yields~\eqref{second_fundamental_bound} and thus concludes the proof.

\appendix

\section{Appendix}
\label{sxn:appendix}

Here, we prove Lemma~\ref{asymmetry_lower_bound}. 
A straightforward calculation reveals that 
\begin{equation*} 
\frac{\partial}{\partial s} \widebar{\psi}(r,s) = \E\left[\langle xx^*\rangle\right], 
\quad \mbox{and} \quad 
\frac{\partial^2}{\partial s^2} \widebar{\psi}(r,s) = \E\left[x^{* 2}(\langle x^2\rangle-\langle x\rangle^2)\right] >0,
\end{equation*} 
so that $s \mapsto \frac{\partial}{\partial s} \widebar{\psi}(r,s)$ is Lipschitz and strongly convex on any interval, and for all $r\ge 0$. �

Let $\nu = P_{\x}$, and let $\mu$ be the symmetric part of $P_{\x}$, i.e., $\mu(A) = (P_{\x}(A)+P_{\x}(-A))/2$ for all Borel $A \subseteq \R$. We observe that $\nu$ is absolutely continuous with respect to $\mu$, so that the Radon-Nikodym derivative $\frac{\rmd \nu}{\rmd \mu}$ is a well-defined measurable function from $\R$ to $\Rp$ that integrates to one.

\begin{proposition}\label{asymmetry_lower_bound2}
For all $r \ge 0$, we have 
\[\widebar{\psi}(r,r)  -\widebar{\psi}(r,-r) \ge 2\E \left[\left\langle \frac{\rmd \nu}{\rmd \mu}(x) - 1 \right\rangle_{\mu,r}^2\right],\]
where $\langle \cdot \rangle_{\mu,r}$ is the average w.r.t.\ to the Gibbs measure corresponding to the Gaussian channel $y = \sqrt{r} x^* + z$, $x^* \sim \mu$ and $z \sim \normal(0,1)$.
Moreover, if $r>0$, the right-hand side of the above inequality is zero \emph{if and only if} $\mu=\nu$, i.e., the prior $P_{\x}$ is symmetric.
\end{proposition}

Finally, the last statement is given here.
\begin{lemma}\label{monotonicity}
The map $r \mapsto \E \left[\left\langle \frac{\rmd \nu}{\rmd \mu}(x) - 1 \right\rangle_{\mu,r}^2\right]$ is increasing on $\Rp$.
\end{lemma}

\begin{proof}
This is a matter of showing that the derivative of the above function is non-negative. By standard manipulations (Gaussian integration by parts, Nishimori property), the derivative can be written as
\[\E\left[\left\langle x\left(\frac{\rmd \nu}{\rmd \mu}(x) - 1\right)\right\rangle_{\mu,r}^2\right].\]
\end{proof}

%\vspace{.5cm}
\begin{proofof}{Proposition~\ref{asymmetry_lower_bound2}}
The argument relies on a smooth interpolation method between the two measures $\mu$ and $\nu$. Let $t \in [0,1]$ and let $\rho_t = (1-t)\mu + t\nu$. Further, let $r,s \ge 0$ be fixed, and
\[\widebar{\psi}(r,s;t) := \E_{z} \int \left(\log \int \exp\left(\sqrt{r}zx + s xx^* - \frac{r}{2} x^2\right) \rmd \rho_t(x)\right) \rmd \rho_t(x^*),\]
where $z \sim \normal(0,1)$. Now let 
\[\phi(t) = \widebar{\psi}(r,r;t) - \widebar{\psi}(r,-r;t).\]
We have $\phi(1) = \widebar{\psi}(r,r) - \widebar{\psi}(r,-r)$ on the one hand, and since $\mu$ is a symmetric distribution, $\phi(0) = 0$ on the other. 
We will show that $\phi$ is a convex increasing function on the interval $[0,1]$, strictly so if $\mu \neq \nu$, and that $\phi'(0) = 0$. Then we  deduce that $\phi(1) \ge \frac{\phi''(0)}{2}$, allowing us to conclude. First, we have
\begin{align*}
\frac{\rmd}{\rmd t} \widebar{\psi}(r,r;t) &= \E_{z} \int \log \int e^{\sqrt{r}zx + r xx^* - \frac{r}{2} x^2} \rmd \rho_t(x) ~\rmd (\nu - \mu)(x^*)\\
&~~~+\E_{z} \int \frac{\int e^{\sqrt{r}zx + r xx^* - \frac{r}{2} x^2}\rmd(\nu - \mu)(x)}{\int e^{\sqrt{r}zx + r xx^* - \frac{r}{2} x^2}\rmd \rho_t(x)} ~\rmd \rho_t(x^*),
\end{align*}
and
\begin{align*}
\frac{\rmd^2}{\rmd t^2} \widebar{\psi}(r,r;t) &= 2\E_{z} \int \frac{\int e^{\sqrt{r}zx + r xx^* - \frac{r}{2} x^2}\rmd(\nu - \mu)(x)}{\int e^{\sqrt{r}zx + r xx^* - \frac{r}{2} x^2}\rmd \rho_t(x)} ~\rmd (\nu - \mu)(x^*)\\
&~~~ -2\E_{z} \int \left(\frac{\int e^{\sqrt{r}zx + r xx^* - \frac{r}{2} x^2}\rmd(\nu - \mu)(x)}{\int e^{\sqrt{r}zx + r xx^* - \frac{r}{2} x^2}\rmd \rho_t(x)}\right)^2 ~\rmd \rho_t(x^*).
\end{align*}
Similar expressions holds for $\widebar{\psi}(r,-r;t)$ where $x^*$ is replaced by $-x^*$ inside the exponentials. We see from the expression of the first derivative at $t=0$ that $\widebar{\psi}(r,r;0)' = \widebar{\psi}(r,-r;0)'$. This is because $\rho_0 = \mu$ is symmetric about the origin, so a sign change (of $x$ for the first term, and $x^*$ for the second term in the expression) does not affect the value of the integrals. Hence $\phi'(0) = 0$. 
Now, we focus on the second derivative. Observe that since $\mu$ is the symmetric part of $\nu$, $\nu - \mu$ is anti-symmetric. This implies that the first term in the expression of the second derivative changes sign under a sign change in $x^*$, and keeps the same modulus. As for the second term, a sign change in $x^*$ induces integration against $\rmd \rho_t(-x^*)$. Hence we can write the difference $(\widebar{\psi}(r,r;t)-\widebar{\psi}(r,-r;t))''$ as  
\begin{align*}
\frac{\rmd^2}{\rmd t^2} \phi(t) &= 4 \E_{z} \int \frac{\int e^{\sqrt{r}zx + r xx^* - \frac{r}{2} x^2}\rmd(\nu - \mu)(x)}{\int e^{\sqrt{r}zx + r xx^* - \frac{r}{2} x^2}\rmd \rho_t(x)} ~\rmd (\nu - \mu)(x^*)\\ 
&~~~-2\E_{z} \int \left(\frac{\int e^{\sqrt{r}zx + r xx^* - \frac{r}{2} x^2}\rmd(\nu - \mu)(x)}{\int e^{\sqrt{r}zx + r xx^* - \frac{r}{2} x^2}\rmd \rho_t(x)}\right)^2 ~(\rmd \rho_t(x^*)-\rmd \rho_t(-x^*)).
\end{align*}
For any Borel $A$, we have $\rho_t(A)- \rho_t(-A) = (1-t)(\mu(A) - \mu(-A)) + t(\nu(A)-\nu(-A)) = 2t(\nu - \mu)(A)$. Therefore the second term in the above expression becomes
\[-4t\E_{z} \int \left(\frac{\int e^{\sqrt{r}zx + r xx^* - \frac{r}{2} x^2}\rmd(\nu - \mu)(x)}{\int e^{\sqrt{r}zx + r xx^* - \frac{r}{2} x^2}\rmd \rho_t(x)}\right)^2 ~\rmd(\nu - \mu)(x^*).\]
Since both $\mu$ and $\nu$ are absolutely continuous with respect to $\rho_t$ for all $0 \le t < 1$ we write
\[\frac{\rmd^2}{\rmd t^2} \phi(t) = 4 \E_{z,x^*} \left\langle \frac{\rmd(\nu - \mu)}{\rmd\rho_t}(x) \frac{\rmd(\nu - \mu)}{\rmd\rho_t}(x^*) \right\rangle - 4 t\E_{z,x^*} \left\langle \frac{\rmd(\nu - \mu)}{\rmd\rho_t}(x)\right\rangle^2,\]
where the Gibbs average is with respect to the posterior of $x$ given $z,x^*$ under the Gaussian channel $y = \sqrt{r}x^* + z$, and the expectation is under $x^* \sim \rho_t$ and $z \sim \normal(0,1)$. By the Nishimori property, we simplify the above expression to
\[\frac{\rmd^2}{\rmd t^2} \phi(t) = 4(1-t)\E \left[\left\langle \frac{\rmd(\nu - \mu)}{\rmd\rho_t}(x)\right\rangle^2\right],\]
where the expression is valid for all $0\le t <1$. From here we see that the function $\phi$ is convex on $[0,1]$ (where we have closed the right end of the interval by continuity). Since $\phi(0) = \phi'(0) = 0$, $\phi$ is also increasing on $[0,1]$. Therefore we have
\[\phi(1) \ge \half\phi''(0) = 2\E \left[\left\langle \frac{\rmd \nu}{\rmd\mu}(x)-1\right\rangle_{\mu,r}^2\right].\]    

Now it remains to show that if $\widebar{\psi}(r,r) = \widebar{\psi}(r,-r)$ for some $r >0$ then $\mu = \nu$. By the lower bound we have shown, equality of $\widebar{\psi}(r,r)$ and $\widebar{\psi}(r,-r)$ would imply
\[\left\langle \frac{\rmd \nu}{\rmd\mu}(x)\right\rangle_{\mu,r} = \frac{\int e^{\sqrt{r}zx + r xx^* - \frac{r}{2} x^2}\rmd\nu(x)}{\int e^{\sqrt{r}zx + r xx^* - \frac{r}{2} x^2}\rmd\mu(x)} = 1\] 
for (Lebesgue-)almost all $z$ and $P_{\x}$-almost all $x^*$. We make the change of variable $z \mapsto \sqrt{r}(z-x^*)$ and complete the squares, then the above is equivalent to
\[\int e^{- \frac{r}{2} (x-z)^2}\rmd\nu(x) = \int e^{- \frac{r}{2} (x-z)^2}\rmd\mu(x)\] 
for almost all $z$. The above expressions are convolutions of the measures $\nu$ and $\mu$ against the Gaussian kernel. By taking the Fourier transform on both sides and using Fubini's theorem, we get equality of the characteristic functions of $\mu$ and $\nu$: for all $\xi \in \R$,
\[\int e^{\imnb \xi x}\rmd \nu(x) = \int e^{\imnb \xi x}\rmd \mu(x).\]
This is because the Fourier transform of the Gaussian (another Gaussian) vanishes nowhere on the real line, thus it can be simplified on both sides.   
This of course implies that $\nu= \mu$, and concludes our proof.
\end{proofof}

\vspace{5mm}
\textbf{Acknowledgments.} 
AE is grateful to L\'{e}o Miolane for insightful conversations. This research effort was initiated at the \emph{Workshop on Statistical physics, Learning, Inference and Networks} at Ecole de Physique des Houches, winter 2017. FK acknowledges funding from the EU (FP/2007-2013/ERC grant agreement 307087-SPARCS). MJ acknowledges the support of the Mathematical Data Science program of the Office of Naval Research under grant number N00014-15-1-2670.

%\vspace{-5mm}

\begin{small}

\bibliographystyle{apalike}
\bibliography{phase_transitions}

\end{small}

\end{document}